\documentclass{article}
\usepackage{amsfonts}
\usepackage{amsmath} 
\usepackage{amssymb}
\usepackage{graphicx}
\usepackage{amsthm}
\usepackage{mathtools}
\usepackage{bbm}
\usepackage{mathrsfs}
\usepackage{hyperref}

\usepackage{dirtytalk}
\usepackage{fullpage}
\usepackage{textgreek}
\usepackage{tikz-cd}
\usepackage{shuffle}
\usepackage{xcolor}
\usepackage{euscript}
\usepackage{ebgaramond}
\usepackage{scalerel}
\usepackage[T1]{fontenc}
\usepackage[ttdefault=true]{AnonymousPro}

\DeclareFontFamily{U}{bigshuffle}{}
\DeclareFontShape{U}{bigshuffle}{m}{n}{
<5-8> s*[1.6] shuffle7
<8->  s*[1.4] shuffle10
}{}
\DeclareSymbolFont{BigShuffle}{U}{bigshuffle}{m}{n}
\DeclareMathSymbol\bigshuffle{\mathop}{BigShuffle}{"001}
\DeclareMathSymbol\bigcshuffle{\mathop}{BigShuffle}{"002}

\def\d{\mathrm{d}}
\newcommand{\var}[1]{{#1 \mathchar"2D \mathrm{var}}}

\def\eps{\varepsilon}
\newcommand{\bX}{\boldsymbol{X}}

\def\Sig{\EuScript{S}}

\def\p{{\lfloor p \rfloor}}

\def\bbR{\mathbb{R}}
\def\cE{\mathcal{E}}
\def\cM{\mathcal{M}}

\newcommand{\vertiii}[1]{{\left\vert\kern-0.25ex\left\vert\kern-0.25ex\left\vert #1 
\right\vert\kern-0.25ex\right\vert\kern-0.25ex\right\vert}}

\numberwithin{equation}{section}
\usepackage{thmtools}
\declaretheorem[numberwithin=section]{theorem}
\declaretheorem[sibling=theorem]{proposition}
\declaretheorem[sibling=theorem]{lemma}

\theoremstyle{definition}

\declaretheorem[sibling=theorem]{definition}
\declaretheorem[sibling=theorem]{problem}
\theoremstyle{remark}
\declaretheorem[sibling=theorem]{remark}

\title{Concise $(\eps,r)$-representations of a path}

\author{Emilio Ferrucci\thanks{SISSA, Trieste; email: \href{mailto:emilio.ferrucci@sissa.it}{emilio.ferrucci@sissa.it}} \and Oliver Perrée\thanks{Work carried out while at the University of Oxford; email: \href{mailto:oliverperree@gmail.com}{oliverperree@gmail.com}} \and Terry Lyons\thanks{University of Oxford and Imperial College London; email: \href{mailto:terry.lyons@maths.ox.ac}{terry.lyons@maths.ox.ac.uk}}}

\usepackage{euscript}
\allowdisplaybreaks[2]

\begin{document}

\maketitle

\begin{abstract}
Paths $X \colon [0,T] \to \mathbb R^d$ are traditionally stored in finite memory as time series. Recent research has underscored the benefits of instead representing them as collections of iterated integrals $\{\int_{0 < u_1 < \ldots < u_n < T} \d X_{u_1} \otimes \cdots \otimes \d X_{u_n}\}_{n = 0}^N$. These two encodings can be viewed as the extrema on a two-parameter spectrum of representations of the path as degree-$N$ signatures on $m$ intervals in a partition of $[0,T]$. We ask the question of which such representation takes up the least amount of memory, measured as number of real values needed to store the truncated log-signature, subject to the constraint of it being able to approximate solutions to linear controlled differential equations (CDEs) $\d Y = AY\d X$ with $|A| \leq r$ at accuracy at least $\eps$. Estimating the error in terms of the length of $X$, we find that the optimal representation generally lies strictly in between the two naive choices $N = 1$ or $m = 1$, and derive its asymptotics as $r \to \infty$ and $\eps \to 0^+$. Similar considerations can be made when estimating the error in terms of the $p$-variation norm of $X$: in this regime we prove an error bound of the degree-$N$ Euler scheme for linear CDEs with decay in both $m$ and (factorially) in $N$ with the other arbitrarily fixed. We conclude by setting up the analogous problem for SDEs, with the error measured in $L^2$, and derive a similar $L^2$-Euler error estimate for It\^o SDEs with drift. We include an empirical study of the optimisation problem, which we demonstrate for toy examples of $p$-rough paths and for fractional Brownian motion.
\end{abstract}

\section*{Introduction}

Sequential data is ubiquitous and generated at an ever increasing rate. In this paper, we consider the problem of how to represent such data so that it takes up the least amount of memory, while still requiring it to be sufficiently expressive. In order to state this problem precisely, we must describe how we are treating data and memory, which we call \emph{storage}, as well as what type of problems we wish to solve, and what it means to solve them in a satisfactory manner.

Beginning with the former, we regard a sequential data point to be a continuous \emph{path} $X \colon [0,T] \to \mathbb R^d$. The usual way of finitely representing $X$ is as a time series $(X_{t_0}, \ldots, X_{t_m})$ with $0 \leq t_0 < \ldots < t_m \leq T$ (usually $t_0 = 0$ and $t_m = T$). For signals that are too oscillatory, however, it has been understood since \cite{Lyo98} that such a description is inadequate: more specifically, unless $X$ is of finite $2 > p$-variation, there is no canonical meaning of the \emph{controlled differential equation} (CDE) $\d Y = V(Y)\d X$. The notion that is instead rich enough to define the meaning of what it means to solve such an equation is that of a \emph{rough path} $\bX$, i.e.\ a collection of tensors indexed by $0 \leq s < t \leq T$ whose levels intuitively represent Chen's \cite{Chen54} iterated integrals
\begin{equation}\label{eq:rp}
\bX_{s,t}^n  \text{\say{$=$}} \int_{s < u_1 < \ldots < u_n < t} \d X_{u_1} \otimes \cdots \otimes \d X_{u_n} \quad\in (\mathbb R^d)^{\otimes n};
\end{equation}
here the equality is meant as a definition (in a precise sense) since these integrals are generally not well defined. If $X$ is of finite $p$-variation, generally $\p$ such levels are required to give a unique meaning to CDEs, and there exist corresponding degree-$\p$ (or higher) numerical schemes converging to the solution.

There is no reason to stop at the first $n = \p$ for which $X$ has finite $p$-variation, and one may consider the \emph{signature} $\Sig(X)$, i.e.\ the collection of iterated integrals \eqref{eq:rp} for $n \in \mathbb N$, truncated at arbitrarily high order. Even on the single interval $[s,t] = [0,T]$, this has the property of characterising the path up to \say{retracings}, see \cite{chen_integration_1958}, \cite{HL10} and \cite{siguniqueness}. Moreover, the Stone-Weierstrass theorem implies that functions can be approximated \emph{linearly} in terms of the signature: roughly speaking, for any continuous function $F$ on a compact set in an appropriate topology of (rough) paths $K$ and any $\varepsilon > 0$ there exists a truncation level $N$ and a linear function $\lambda_N$ of $\Sig_N(X)_{0,T}$ (the signature truncated at degree $N$, see \eqref{eq:sig} below), such that 
\begin{equation}\label{eq:SW}
\sup_{X \in K}|F(X) - \langle \lambda_N, \Sig_N(X)_{0,T} \rangle| < \varepsilon,
\end{equation}
see e.g.\ \cite[Proposition 1]{Ferm21}. A growing body of literature initiated in the research program \cite{Lyo14} uses these facts as motivation to consider the signature as a feature set for sequential data in machine learning. However, \eqref{eq:SW} has the drawback of not being quantitative (which would be difficult without being more specific about $F$), and there are well-known instances \cite{Tref23} of a \say{universal approximation} theorem guaranteeing a dense span while \say{the associated expansions are so inefficient as to have no conceivable relevance to any actual computation} (the example in question concerns monomials of even degree in $C[0,1]$). Very recently this shortcoming has been addressed in \cite{horvath} in the context of statistical learning.

While the motivation behind the signature method in machine learning hinges on universality in the degree of the signature $N \to \infty$, the literature on numerical methods (e.g.\ \cite{KP10}) is instead focused on low degree signatures and proving convergence in the number of intervals $m \to \infty$ (but see also \cite{MFKL21} for an empirical evaluation of the signature method over multiple intervals). In this paper we interpolate between the two and represent our data as collections of degree-$N$ signatures on $m$ sub-intervals. Assuming the original data $X$ was discretised on a much finer grid, the encoding \eqref{eq:msigs} can be viewed as a form of lossy compression as long as $N$ is not too large. In fact, the signatures $\Sig_N(X)_{t_{k-1},t_k}$ can be replaced with their corresponding log-signatures, free Lie algebra-valued elements that quotient out the integration-by-parts relations present in $\Sig(X)$, resulting in mild but lossless gains. The encoding then becomes
\begin{equation}\label{eq:msigs}
X \ \longleftarrow \ (\log\Sig_N(X)_{t_0,t_1}, \ldots, \log\Sig_N(X)_{t_{m-1}, t_m}).
\end{equation}
Storage cost is computed naively as the number of reals that need to be stored in memory. While this is how it is usually defined in other articles, e.g.\ \cite{Fef09}, it leaves out the important question how information is further lost to floating point arithmetic (and how these errors propagate).

Since the \eqref{eq:msigs} tends to become fully descriptive as either $N$ or $m$ approaches $\infty$, we must agree on a class of problems to solve in order to trade $N$ optimally off against $m$: we consider \emph{linear CDEs}
\begin{equation}\label{eq:linearCDE}
\d Y = A Y \d X,
\end{equation}
which figure prominently in state space models like Mamba \cite{mamba}, see \cite{COWSL24, walker2025structured, walkerPhD} for the link with rough path theory. Linear CDEs are not all equivalent in complexity, and in fact it is natural to parametrise them in terms of the tensor norm $|A|$. In order to state the problem, it is necessary to only consider problems \eqref{eq:linearCDE} with $|A| \leq r$ for some $r$ (it is also necessary set a bound on the length of $X$ or proxy thereof, which we omit in this preliminary description of the problem). Finally, since it is impossible to solve any equation exactly, we must agree on an acceptable tolerance $\varepsilon$.

\begin{problem}[informal]\label{prob}
Find a choice of $(N, m)$ that makes the number of scalars in the encoding \eqref{eq:msigs} as low as possible, while being sufficiently rich to solve \eqref{eq:linearCDE} for all $A$ of norm bounded by $r$ at accuracy at least $\varepsilon$.
\end{problem}

Approximations will all be carried out in terms of degree-$N$ Euler schemes, see \cite[Ch.\ 10]{FV10}. We will call an encoding that satisfies such a property $(\varepsilon, r)$-close to the original $X$. This notion is closely related to the idea of Kolmogorov $\varepsilon$-entropy/complexity \cite{Kol56, KT59, Kol83, DP13}.

We would also like to mention the related work \cite{BBL23}. Here the authors consider the similar problem of deciding whether to refine the partition or raise the log-signature degree, in the context of an adaptive log-ODE solver and with the objective of reducing computational work. Their algorithm is adaptive and problem-dependent (e.g.\ the solver can pick different log-signature degrees $N$ over different intervals, something not considered here), and works for non-linear RDEs. We instead seek representations of paths that work uniformly over the simpler class of norm-bounded linear equations. This allows us to study the tradeoff between partition size and truncation degree independently of any particular equation, and to focus on its asymptotic behaviour as well as on special features of numerical schemes for linear equations.

Our contributions are as follows. In \autoref{sec:1var} we treat \autoref{prob} while measuring errors in terms of the length ($1$-variation) of $X$. In this case, sharp estimates for the Euler error are easy to obtain and continuously relaxing $N, m$ results in a constrained optimisation problem which we prove has a unique solution. In so doing, we characterise the asymptotic behaviour of the dimension of the free Lie algebra over $d$ generators, i.e.\ the storage taken up by a single log-signature. We study the constrained optimisation problem in the asymptotic regimes of $\varepsilon \to 0^+$ and $r \to \infty$. We show that in both cases the optima $N^*$ and $m^*$ are both pushed to $\infty$, at rates which we derive. This is significant for $N^*$, since it shows that despite the log-signature being cursed by dimensionality, it is beneficial to store it at higher degrees for equations of higher norm and/or when requiring a lower tolerance. More specifically, in \autoref{thm:asympN*m*} we show that one can expect
\[
N^* \approx \sqrt{\frac{\log(e^rr/\eps)}{\log d}}, \qquad m^* \approx \bigg(\frac{e^r r^{N^*+1}}{\eps (N^*+1)!} \bigg)^{1/N^*}
\]
and also observe that the corresponding storage footprint grows more slowly than it would given the naive choices $N = 1$ and $m$ as large as needed. A more computational study of this optimisation problem corroborates the main results, additionally validating claims that are difficult to show analytically (e.g.\ that also $m = 1$ is asymptotically suboptimal).

In \autoref{sec:pvar} we begin by deriving an error estimate for the degree-$N$ Euler scheme in which the error is quantified in terms of the $p$-variation of the rough path $\Sig_\p(X)$, see \autoref{thm:eulerp}. We insist on the form of this error being similar to that in the $1$-variation case, and in particular that it decay not only polynomially in $m$ but factorially in $N$, a special feature of linear CDEs (simpler proofs would exist without this requirement). We regard this result to be of independent interest, since it shows convergence of the Euler scheme not only as $m \to \infty$ but also as $N \to \infty$ with $m$ arbitrary, including $m = 1$, independently of the other parameters. Formulating the corresponding constrained optimisation problem in a convenient manner is more challenging in the $p$-variation case, because the parameter $p$ can be chosen freely in $[1,N+1)$ and the formulation is in terms of the \emph{$p$-variation curve} $p \mapsto \|X\|_{\var{p}}$ of $X$, an infinite-dimensional quantity. We instead advocate for choosing the optimal $N$ and $m$ by solving batches of RDEs and measuring errors empirically. We put this plan into action by considering a class of smooth paths that approximate pure degree-$\nu \in \mathbb N$ rough paths, which underscore the importance of being able to store the path with $N > 1$. The important takeaway message of this section is that rough path theory is of value even when applied to smooth paths (see also \cite{BFPP22, maud2}): in the problem considered here, this is due to the fact that it unlocks the possibility of estimating errors according to norms that are more favourable.

Finally, in \autoref{sec:random} we set up the analogous problem for paths sampled from a stochastic process, with the error measured in $L^2$. Using martingale arguments and properties of the signature of Brownian motion with drift, in \autoref{thm:L2error} we derive an $L^2$-Euler error estimate for SDEs which has similar features to those in previous sections. The corresponding constrained optimisation problem could be studied and is qualitatively very similar to that in the bounded variation case; we instead opt for a practical approach here too, which we demonstrate with an empirical study of the case of multidimensional fractional Brownian motion with Hurst parameter $H > \frac 14$. Our results once again show that it is often beneficial to pick $N$ significantly higher than $1$ or even than the minimum $\lfloor H^{-1} \rfloor$ at which the rough path is defined, including when $H > \frac 12$.

While motivated by practical questions (e.g.\ see \autoref{fig:ukdale-cycle-patterns}), we believe our results reveal something quite deep about the nature of approximate sequential data: in our framework, the intrinsic roughness and dimensionality of a path is only defined w.r.t.\ a reference choice of $\eps$ and $r$. We hope this perspective can be developed further in the stochastic case: while by Chow's theorem \cite[\S 7.4]{FV10} all truncated Lie elements can be realised as truncated log-signatures of paths, this leaves out the important fact that, when $X$ is drawn from a distribution on path-space, the coordinates of its log-signature (e.g.\ in a Hall basis) may have a lot of redundancy. We believe it would be interesting to investigate optimal encodings of such data subject to similar $(\varepsilon, r)$-constraints, using ideas from information theory, and to strengthen the link with Kolmogorov complexity.

\begin{figure}[h!]
	\centering
	\begin{minipage}[c]{0.48\textwidth}
		\centering
		\includegraphics[width=\linewidth]{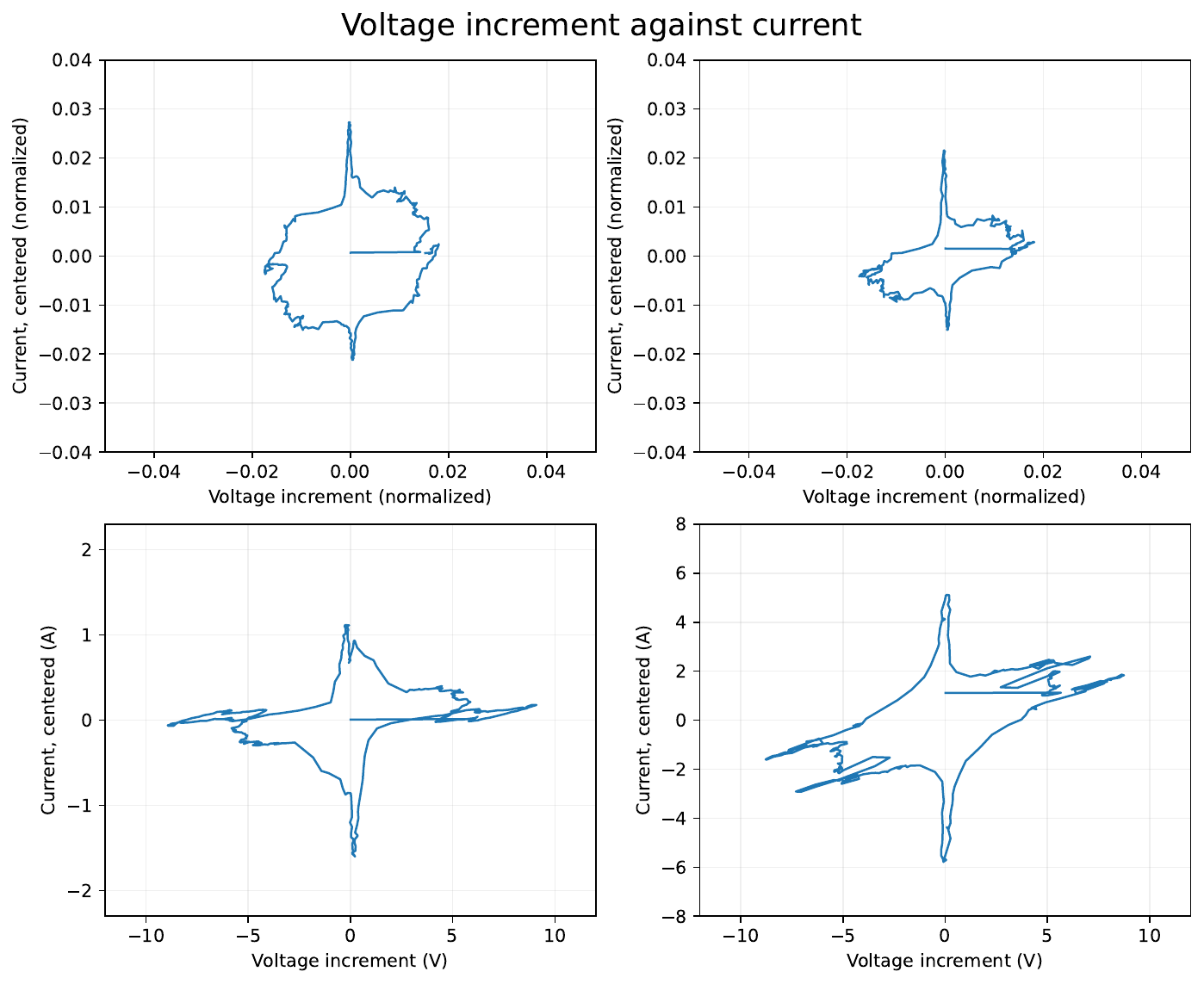}
	\end{minipage}\hfill
	\begin{minipage}[c]{0.48\textwidth}
			\caption{Four one-cycle voltage-increment/current patterns from the high-frequency UK-DALE domestic electricity dataset \cite{kelly2015ukdale}; a similar dataset has previously been analysed using signatures \cite{moore2022path}. The cycles were selected from three two-channel 16~kHz voltage/current FLAC recordings from houses~1, 2, and~5, totalling 489~MiB (more than 0.5~GB). The plots expose geometric structure in the aggregate electrical state, related to the type of home and combination of appliances in use. For datasets like these, truncated logsignatures over appropriately sized intervals have the potential benefit of providing much more compact and informative summaries of the path geometry compared to long raw time series. }
		\label{fig:ukdale-cycle-patterns}
	\end{minipage}
\end{figure}

\paragraph{Acknowledgements.}\ \\[-2ex]

We would like to thank Jason Rader for fruitful discussions.

During the first phase of this project, EF was employed at the University of Oxford and supported through the EPSRC programme grant [EP/S026347/1]. EF's research is currently supported through the ERC Grant SQGT [101116964]. TL was supported in part by UK Research and Innovation (UKRI) through the Engineering and Physical Sciences Research Council (EPSRC) via Programme Grants [Grant No. UKRI1010: High order mathematical and computational infrastructure for streamed data that enhance contemporary generative and large language models], [Grant No. EP/S026347/1: Unparameterised multi-model data, high order signatures and the mathematics of data science], and the UKRI AI for Science award [Grant No. UKRI2385: Creating Foundational Benchmarks for AI in Physical and Biological Complexity]. He was also supported by The Alan Turing Institute under the Defence and Security Programme (funded by the UK Government) and through the provision of research facilities; by the UK Government; and through CIMDA@Oxford, part of the AIR@InnoHK initiative funded by the Innovation and Technology Commission, HKSAR Government.

For the purpose of open access, the authors have applied a Creative Commons Attribution (CC BY) licence to any author accepted manuscript version arising from this manuscript.

\section{Paths of bounded variation}\label{sec:1var}

We refer to \cite{FV10} for the following background notions. Let $X \in C^\var{1}([0,T],\mathbb R^d)$. We denote the (truncated) \emph{signature} of $X$ by
\begin{equation}\label{eq:sig}
	\begin{split}
		\Sig^{n}(X)_{s,t} &\coloneqq \int_{s < u_1 < \ldots < u_n < t} \d X_{u_1} \otimes \cdots \otimes \d X_{u_n} \quad\in (\mathbb R^d)^{\otimes n} , \\
		\qquad \Sig_{N}(X)_{s,t} &\coloneqq \sum_{k = 0}^n \Sig^{k}(X)_{s,t} \quad\in T_N(\mathbb R^d) \coloneqq \bigoplus_{k = 0}^N (\mathbb R^d)^{\otimes k}
	\end{split}
\end{equation}
We will write
\begin{equation}
	\pi^n \colon T(\bbR^d) \twoheadrightarrow (\bbR^d)^{\otimes n},\qquad \pi_N \colon T(\bbR^d) \twoheadrightarrow T_N(\bbR^d).
\end{equation}
for the canonical projections (where $T(\bbR^d)$ is the untruncated tensor algebra), and recall that
\begin{equation}\label{eq:dimT}
	\dim(T_N(V)) = \frac{d^{N+1} - 1}{d-1} .
\end{equation}
The algebraic (shuffle) relations in $\Sig_N(X)$ can be eliminated by expressing it as exponential of a Lie series, the (truncated) \emph{log-signature}
\begin{align*}
	\pi_N \log \Sig(X) \in \mathfrak L_N(\bbR^d), \qquad \log(x) = \sum_{n \geq 1} \frac{(-1)^{n-1}}{n} (x - 1)^{\otimes n}
\end{align*}
valued in the free Lie algebra over $d$ generators truncated at degree $N$, $\mathfrak L_N(\bbR^d)$; the signature can be recovered by exponentiating the log-signature (noncommutatively, in the tensor algebra). This motivates our interest in
\begin{equation}
\Lambda_d (N) \coloneqq \dim \mathfrak L_N (\bbR^d),
\end{equation}
which we regard as the intrinsic dimension of the degree-$N$ data of $X$ (modulo tree-like equivalence).

\begin{remark}
	Here and below it is only interesting to take $d \geq 2$, since otherwise there are no non-trivial Lie brackets and $\mathfrak L (\bbR^d) = \bbR^d$. The convergence of Euler schemes presented below does not assume it, but time-homogeneous controlled differential equations are only an interesting (i.e.\ characteristic) class of functions on paths if the path is genuinely multidimensional. This is not limiting, since it is always possible to augment a path with the time coordinate, $t \mapsto (t, X_t)$. We therefore make the standing assumption that $d \geq 2$ (except for in \autoref{sec:random} in which time is automatically prepended).
\end{remark}

Let $A \in \mathcal L(\mathbb R^d, \mathcal L(\mathbb R^e, \mathbb R^e))$, with $d \geq 2$ and $e \geq 1$ integers. We consider the linear controlled differential equation (CDE)
\begin{equation}\label{eq:linear}
\d Y_t^j = A^j_{i\gamma} Y_t^i \d X^\gamma_t, \quad Y_0 = y_0 .
\end{equation}
in which we have used the Einstein summation convention. We view this as a function on paths, i.e.\ write
\begin{equation}
\Phi^A_{s,t} \colon C^\var{1}([s,t],\mathbb R^d) \to \mathcal L(\mathbb R^e, \mathbb R^e) , \quad X \mapsto (Y_s \mapsto Y_t) .
\end{equation}
which we will often abbreviate $\Phi_{s,t}$ when there is no ambiguity as to the choice of $X, A$. It satisfies the semigroup property
\[
\Phi_{s,t} = \Phi_{u,t}   \Phi_{s,u},\qquad s < u < t.
\]
and is given explicitly by the series \cite[Theorem 4.5]{LCL}
\begin{equation}\label{eq:AyS}
\Phi^A_{s,t}(X)y = \sum_{n = 0}^\infty A^{\otimes n} y \Sig^n(X)_{s,t} ,
\end{equation}
where we denote
\[
A^{\otimes n} = (A^i_{j_1\gamma_1} A^{j_1}_{j_2\gamma_2} \cdots A^{j_{n-1}}_{j\gamma_n})^i_{j, \gamma_1\ldots \gamma_n} \quad \in \mathcal L((\bbR^d)^{\otimes n}, \mathcal L(\mathbb R^e, \bbR^e)) .
\]
Define, the one-step degree-$N$ \emph{Euler scheme} and its iteration over a partition $\pi = \{0 = t_0 <\ldots < t_m = T\}$
\begin{equation}
\mathcal E^A_{N;s,t}(X)y \coloneqq \sum_{n = 0}^N A^{\otimes n} y \Sig^n(X)_{s,t}, \qquad \mathcal E^A_N(X)_{\pi} \coloneqq \mathcal E^A_{N;t_{m-1},t_m}(X) \cdots \mathcal E^A_{N;t_0,t_1}(X)
\end{equation}
Notice that both $\Phi_{s,t}$ and $\cE_{N;\pi}$ are linear in the initial condition $y$. Given a partition $\pi$ as above, we will often apply the Euler scheme to the restricted partition $\pi[t_0,t_k] \coloneqq \{t_0,\ldots, t_k\}$ and we will use similar abbreviations as for $\Phi$.

Throughout this paper we will use the Hilbert-Schmidt inner product and norm on $(\mathbb R^d)^{\otimes n}$, given by
\[
\langle x_1 \otimes \cdots \otimes x_n , y_1 \otimes \cdots \otimes y_n \rangle = \langle x_1, y_1 \rangle \cdots \langle x_n, y_n \rangle
\]
on elementary tensors, and extended linearly. By Cauchy-Schwarz it follows that, for $S \in (\bbR^d)^{\otimes n}$
\begin{equation}\label{eq:AySleq}
|A^{\otimes n}yS| \leq |A|^n |y| |S|,
\end{equation}
where the norm on $A$ is also Euclidean, i.e.\ Frobenius. It would be possible and even more natural to use other combinations of norms on $A$ and $S$, e.g.\ injective or projective (see \cite[Appendix 5.6]{LQ02}), but these have the potential drawback of being harder to evaluate practically.

In this paper, we will be concerned with simplifications and variations of following optimisation problem, which seeks to minimise the storage taken up by the data needed to run the Euler scheme $\mathcal E_{N, \pi}$ with $\pi$ containing $m$ intervals, subject to $N$ and $\pi$ being chosen so that such approximation lies within $\varepsilon$ of the solution for all differential equations of norm less than $r$: 
\begin{equation}\label{eq:problemBasic}
\begin{split}
\text{minimise}\quad & m\Lambda_d(N) \quad \text{over } N, m \in \mathbb N^+ \\
\text{subject to}\quad &|\Phi_{0,T}^A(X) - \mathcal E_N^A(X)_\pi | \leq \varepsilon \quad \text{for all } A, X \text{ with } |A| \|X\|_{\var{1}} \leq r.
\end{split}
\end{equation}
Here and below the norm $|\Phi - \mathcal E_{N,\pi}|$ should be understood as taking a supremum over all $y_0$ on the unit ball. It is necessary to constrain $|A|$ because the behaviour of any numerical scheme can generally expected to be arbitrarily poor as $|A| \to \infty$. The variables over which we are optimising are $N, m$, while the fixed parameters are $d \in \mathbb N_{\geq 2}$, and $\varepsilon, r > 0$. While the optimisation variables $N, m$, will only depend on $d, \eps, r$, the choice of the partition $\pi$ is allowed to depend on $X$.

The problem \eqref{eq:problemBasic} is not analytically tractable. Rather, we will substitute the asymptotics for $\Lambda_d(N)$ as $N \to \infty$ and a bound in terms of $N, m, r$ and $\| X \|_{\var{1}}$ for the error in the constraint. We begin by deriving the former (see \autoref{fig:dims}).

\begin{theorem}\label{thm:dims}
	For $d \geq 2$ and $N \geq 1$
	\[
	\Lambda_d (N) \sim \frac{d^{N+1}}{(d-1)N} \quad \text{as} \quad N \to \infty .
	\]
\end{theorem}

\begin{proof}
Recall Witt's formula \cite{Witt1937} for the dimension of the homogeneous components of the free Lie algebra over $d \geq 2$ generators:
\begin{equation}\label{eq:witt}
	\ell_d(n) \coloneqq \dim \pi_n \mathfrak L(\bbR^d) = \frac 1n \sum_{k \mid n} \mu(k) d^{n/k} = \frac{d^n}{n} + \mathrm O \bigg( \frac{d^{n/2}}{n} \bigg).
\end{equation}
where the sum is over positive divisors of $n$ and $\mu$ is the M\"obius function. Therefore 
\[
\Lambda_d(N) = \sum_{n = 1}^N \ell_d(n) = \sum_{n = 1}^N \frac{d^n}{n} + \mathrm{O}\bigg( \sum_{n = 1}^N \frac{d^{n/2}}{n}  \bigg) = \sum_{n = 1}^N \frac{d^n}{n} + \mathrm{O}(d^{n/2}) \sim \sum_{n = 1}^N \frac{d^n}{n} \eqqcolon \Sigma_d(N).
\]
For this we reproduce the argument of \cite{888521}, slightly generalised. First,
\[
\Sigma_d(N) \geq \frac 1N \sum_{n = 1}^N d^n = \frac 1N \frac{d^{N+1} - d}{d-1} .
\]
Next, observe that for any $u \in (0,1)$,
\begin{align*}
	\Sigma_d(N) &= \sum_{1 \leq n < uN} \frac {d^n}{n} + \sum_{uN \leq n \leq N} \frac {d^n}{n} \\
	&\leq \sum_{1 \leq n < uN} d^n + \frac{1}{uN} \sum_{uN \leq n \leq N} d^n \\
	&\leq \frac{d^{uN} - d}{d - 1} + \frac{1}{uN} \frac{d^{N+1} - d}{d - 1} .
\end{align*}
We have shown
\[
1 - d^{-N} \leq \frac{(d - 1)N}{d^{N+1}} \Sigma_d(N) \leq N (d^{uN - (N+1)} - d^{-N}) + \frac{1 - d^{-N}}{u} ,
\]
from which it follows that 
\[
1 \leq \liminf_{N \to \infty} \frac{(d-1)N}{d^{N+1}} \Sigma_d(N) \leq \limsup_{N \to \infty} \frac{(d-1)N}{d^{N+1}}\Sigma_d(N) \leq \frac 1u
\]
and the assertion follows by taking $u \nearrow 1$.
\end{proof}

\begin{figure}[htbp]
	\centering
	\includegraphics[width=0.6\textwidth]{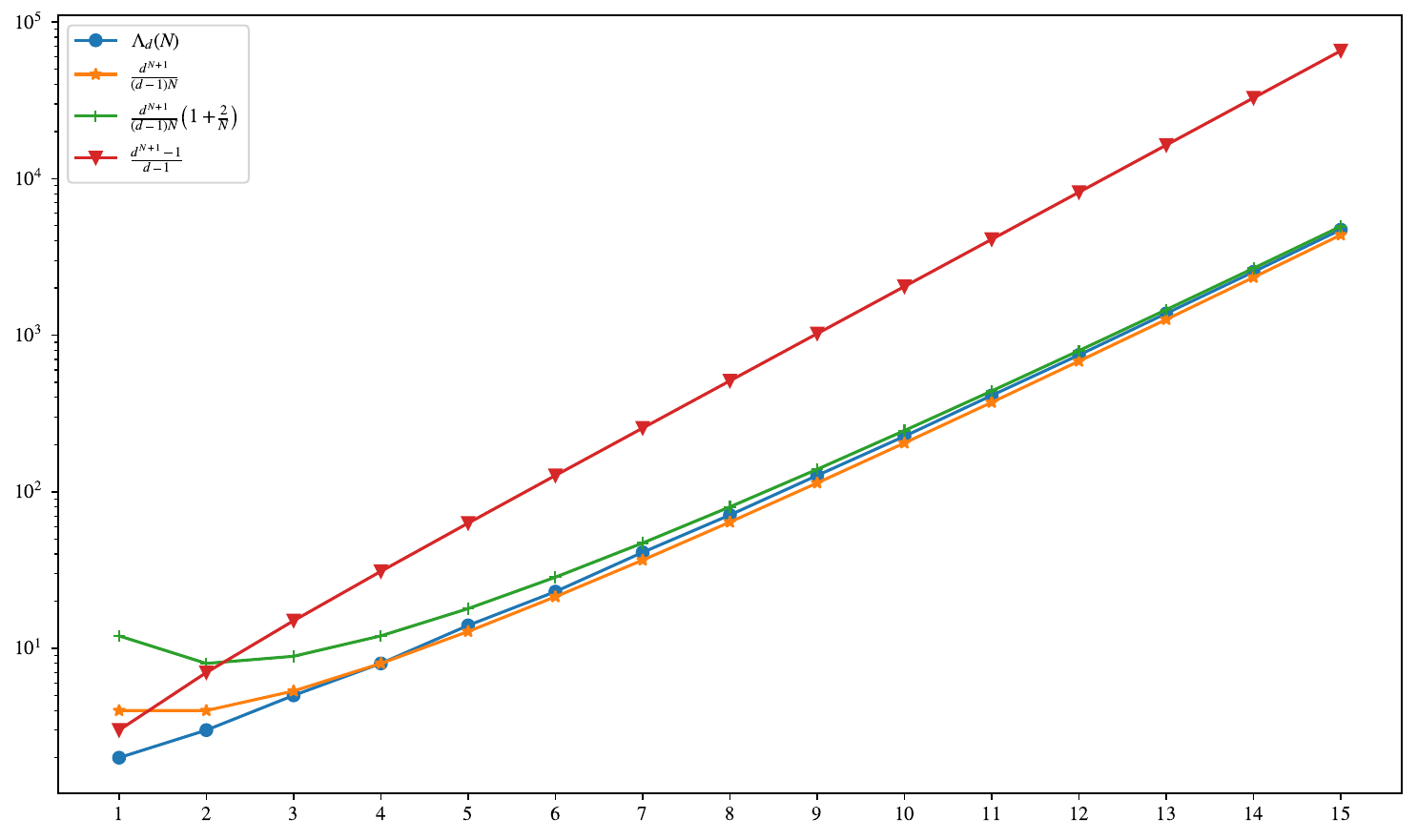}
	\caption{Growth of the dimension of the truncated free Lie algebra, dimension of the tensor algebra, asymptotics predicted by \autoref{thm:dims} for $d = 2$ and $1 \leq N \leq 15$. We have additionally plotted a function which is easy to show is a strict upper bound to $\Lambda_d(N)$ with the same asymptotics.}
	\label{fig:dims}
\end{figure}

We proceed to derive this bound, which is similar to many ones in the literature (e.g.\ it has the expected decay in $m$), but has the feature, specific to linear CDEs, of also decaying factorially in $N$. We refer to the \emph{equal length split} partition $\pi$ to be the (unique, as long as $t \mapsto \| X \|_{\var{1}, [0,t]}$ is strictly increasing) partition of $[0,T]$ with a given number of intervals ($m$) such that $\| X \|_{\var{1}, [t_k,t_{k-1}]} \equiv \| X \|_{\var{1}}/m$.

\begin{theorem}[Euler error for linear CDEs, bounded variation case]\label{thm:euler1}
For any $X \in C^{\var{1}}([0,T], \mathbb R^d)$, $A \in \mathcal L(\mathbb R^d, \mathcal L(\mathbb R^e, \mathbb R^e))$ with $|A| \leq r$, $N, m \in \mathbb N^+$ and $\pi$ a partition of $[0,T]$ with $m$ intervals, we have
\begin{align*}
|\Phi_{0,T}^A(X)y_0 - \mathcal E_{\pi[t_0,t_m]} y_0| &\leq \frac{r^{N+1}}{(N+1)!} \exp(r \| X \|_{\var{1}}) y_0 \sum_{k = 1}^m \| X \|_{\var{1},[t_{k - 1}, t_k]}^{N+1} \\
&= \frac{(r\| X \|_{\var{1}})^{N+1}}{(N+1)! m^N} \exp(r \| X \|_{\var{1}}) |y_0|
\end{align*}
where the equality holds if $\pi$ is the equal length split $\| X \|_{\var{1}, [t_{k-1}, t_k]} \equiv \| X \|_{\var{1}}/m$, the optimal choice among all partitions with $m$ intervals. Finally, this bound is asymptotically sharp as either $m \to \infty$ with $N$ fixed, and as $N \to \infty$ with $m$ fixed, in the latter case up to a multiplicative $m$-dependent constant converging to $1$ as $m \to \infty$.
\end{theorem}
\begin{proof}
Since all quantities are linear in $y_0$ we can eliminate it from the notation. From factorial decay of the signature in the bounded variation case \cite[Proposition 2.2]{LCL}
\[
|\Sig^n(X)_{s,t}| \leq \frac{\|X\|_{\var{1}}^n}{n!},
\]
\eqref{eq:AyS}, and \eqref{eq:AySleq} we have the bounds on the solution and one-step error
\begin{align*}
|\Phi_{s,t}|, |\cE_{N;s,t}| &\leq \sum_{n = 0}^\infty r^n \frac{\|X\|_{\var{1,[s,t]}}^n}{n!} = \exp(r \|X\|_{\var{1},[s,t]} ),\\
|\Phi_{s,t} - \cE_{N;s,t}| &\leq \sum_{n = N+1}^\infty r^n \frac{\|X\|_{\var{1},[s,t]}^n}{n!} \leq \frac{(r \|X\|_{\var{1},[s,t]})^{N+1}}{(N+1)!} \exp(r \|X\|_{\var{1},[s,t]})
\end{align*}
Therefore, using the following simple form of the discrete Grönwall inequality
\begin{equation}\label{eq:discreteGronwall1}
\begin{split}
&\varepsilon_l \leq b_l + e^{a_l} \varepsilon_{l-1}, \quad l \geq 1, \quad \varepsilon_0 = 0 \\
\implies \ &\varepsilon_l \leq \sum_{k = 1}^l b_k \exp \Big( \sum_{j = k + 1}^l a_j \Big)
\end{split}
\end{equation}
for $\eps_l,a_l,b_l$ all positive, and the additivity of $1$-variation over consecutive intervals, we obtain
\begin{align*}
&|\Phi_{t_0,t_l} - \mathcal E_{\pi[t_0,t_l]}| \\
={}& |\Phi_{t_{l-1}, t_l} \Phi_{t_0,t_{l-1}} - \mathcal E_{t_{l-1},t_l} \mathcal E_{\pi[t_0,t_{l-1}]}|\\
\leq{}& |\Phi_{t_{l-1}, t_l} - \cE_{t_{l-1}, t_l}||\Phi_{t_0,t_{l-1}}| + |\mathcal E_{t_{l-1},t_l}| |\Phi_{t_0,t_{l-1}} - \mathcal E_{\pi[t_0,t_{l-1}]}| \\
\leq{}& \frac{(r \|X\|_{\var{1},[t_{l-1},t_l]})^{N+1}}{(N+1)!}\exp(r \|X\|_{\var{1},[t_0,t_l]}) + \exp(r \|X\|_{\var{1},[t_{l-1},t_l]})|\Phi_{t_0,t_{l-1}} - \mathcal E_{\pi[t_0,t_{l-1}]}| \\
\leq{} &\sum_{k = 1}^l \frac{(r \| X\|_{\var{1},[t_{k-1},t_k]})^{N+1}}{(N+1)!}\exp(r \| X\|_{\var{1},[t_0,t_k]}) \exp\Big(r \sum_{j = k+1}^l \| X\|_{\var{1},[t_{j-1},t_j]}\Big) \\
={}& \frac{r^{N+1}}{(N+1)!} \exp(r \| X \|_{\var{1},[t_0,t_l]}) \sum_{k = 1}^l \| X \|_{\var{1},[t_{k - 1}, t_k]}^{N+1}
\end{align*}

The claim about optimality of the bound for the equal-length split follows from the fact that for any $x > 0$ the function
\[
\Big\{(x_1,\ldots, x_m) \in [0,+\infty)^m : \sum_{k = 1}^m x_k = x \Big\} \to [0,+\infty), \qquad (x_1,\ldots,x_m) \mapsto \frac 1m \sum_{k = 1}^m x^{N+1}
\]
is minimised for $x_1 = \ldots = x_m = x/m$, since $x \mapsto x^{N+1}$ is strictly convex, by the condition for equality in Jensen's inequality for convex combinations.

Regarding the sharpness claim, we consider the one-dimensional ODE for $e^{rt}$, $\d Y = rY\d t$. Since $X_t = t$, $\| X \|_{\var{1}} = T$ and we compute directly
\begin{equation}\label{eq:sharp}
\begin{split}
|\Phi_{0,T}^A(X)y_0 - \mathcal E_{\pi[t_0,t_m]} y_0| &= \bigg| e^{rT} - \Big( \sum_{n = 0}^N \frac{(rT/m)^n}{n!} \Big)^m \bigg||y_0| \\
&= e^{rT}\bigg| 1 - \Big( 1 - e^{-rT/m}\sum_{n = N+1}^\infty \frac{(rT/m)^n}{n!} \Big)^m \bigg||y_0|.
\end{split}
\end{equation}
With $N$ fixed and as $m \to \infty$, using that for $a_m \geq 0$ with $ma_m \to 0$, $(1-a_m)^m = 1 - ma_m + \mathrm{O}((ma_m)^2)$,
\begin{align*}
|\Phi_{0,T}^A(X)y_0 - \mathcal E_{\pi[t_0,t_m]} y_0| &= e^{rT}\bigg| 1 - \Big( 1 - (1+\mathrm{O}(m^{-1})) \frac{(rT/m)^{N+1}}{(N+1)!} + \mathrm{O}(m^{-(N+2)}) \Big)^m \bigg||y_0| \\
&\sim e^{rT} \bigg| 1 - \bigg( 1 - \frac{(rT)^{N+1}}{(N+1)!m^N} \bigg) \bigg||y_0| \\
&\sim \frac{e^{rT} (rT)^{N+1}}{(N+1)!m^N}|y_0|.
\end{align*}
With $m$ fixed and $N \to \infty$, starting again from \eqref{eq:sharp}, since for $0 \leq b_N \to 0$, $(1-b_N)^m = 1 - mb_N + \mathrm{O}(b_N^2)$,
\begin{align*}
|\Phi_{0,T}^A(X)y_0 - \mathcal E_{\pi[t_0,t_m]} y_0| &= e^{rT}\bigg| 1 - \Big( 1 - e^{-rT/m} \Big( \frac{(rT/m)^{N+1}}{(N+1)!} + \mathrm{O}\Big(\frac{(rT/m)^{N+2}}{(N+2)!}\Big) \Big) \Big)^m \bigg||y_0| \\
&\sim e^{rT(1- 1/m)} \frac{(rT)^{N+1}}{(N+1)!m^N} |y_0|,
\end{align*}
concluding the proof since
\[
\frac{e^{rT(1- 1/m)}/m^N}{e^{rT}/m^N} = e^{-rT/m} \to 1.\qedhere
\]
\end{proof}

As anticipated, we now rephrase the constrained optimisation problem \eqref{eq:problemBasic} in terms of the storage asymptotics and Euler error bound (with equal length split partition). The form of the bound justifies a posteriori $|A|$ and $\|X\|_{\var{1}}$ being multiplied in \eqref{eq:problemBasic}; without loss of generality we may normalise $X$ to have unit length, reobtaining the formulation of \autoref{prob}. Moreover, we consider the problem in its continuously relaxed form, i.e.\ optimising over $N, m \in (0,\infty)$, with the factorial extended with the Gamma function $(N+1)! = \Gamma(N+2)$ as usual. The problem thus becomes
\begin{equation}\label{eq:relaxed}
\begin{split}
	\text{minimise}\quad &f_d(N,m) \coloneqq m\frac{d^{N+1}}{(d-1)N} \quad \text{over } N, m \in (0,\infty) \\
	\text{subject to}\quad &g_r(N, m) \coloneqq \frac{r^{N+1}e^r}{(N+1)! m^N} = \eps.
\end{split}
\end{equation}
Here we have replaced the inequality $g_r \leq \eps$ with an equality: this is possible since $g_r$ is strictly decreasing in $m$ for every $N$, implying the constraint is active, i.e.\ the minimum will be attained on the boundary of the admissible region.

Let $m^*(N)$ be the value of $m$ which satisfies the constraint for a given $N$, namely
\begin{equation}
m^*(N) \coloneqq \bigg(\frac{e^r r^{N+1}}{\eps (N+1)!} \bigg)^{1/N}
\end{equation}
and consider the function
\begin{equation}\label{eq:phisub}
\phi(N) \coloneqq f_d(N, m^*(N)) = \frac{d^{N+1}}{(d-1)N} \bigg(\frac{e^r r^{N+1}}{\eps (N+1)!} \bigg)^{1/N}, \qquad N \in (0,+\infty).
\end{equation}

\begin{lemma}\label{eq:minphi}
For $\eps \leq e^r r$ the function $\phi$ diverges to $+\infty$ as $N \to 0, +\infty$ and is strictly convex on $(0,+\infty)$. Thus \eqref{eq:relaxed} admits a unique solution.
\end{lemma}

\begin{proof}
Rewriting, we see that for $N \to 0^+$
\[
\phi(N) =  \frac{d^{N+1}}{(d-1)N} \exp\bigg(\frac 1N \log \frac{e^r r^{N+1}}{\eps (N+1)!} \bigg) \xrightarrow{N \to 0^+} +\infty \qquad \text{for } \eps < e^r r
\]
since the argument of the logarithm tends to $e^r r/\eps > 1$. By Stirling's approximation, without any condition on the parameters we have for $N \to +\infty$
\[
\phi(N) \sim \frac{d^{N+1}}{(d-1)N} \bigg(\frac{e^r (er)^{N+1}}{\sqrt{2\pi}\eps (N+1)^{N+3/2}} \bigg)^{1/N} \sim \frac{er}{d-1} \frac{d^{N+1}}{N^2} \xrightarrow{N \to \infty} + \infty
\]
We now show strict convexity by showing strict log-convexity:
\[
\phi''(N) = \phi(N)(\theta''(N) + \theta'(N)^2) > 0 \ \Longleftarrow \ \theta''(N) > 0 \qquad \text{for } N \in (0,+\infty), \qquad \text{with } \theta \coloneqq \log \phi.
\]
Recalling the digamma and polygamma functions and the series representation fo the latter \cite[(6.4.10)]{handbook}
\begin{align*}
	\psi(z) &= \frac{\d}{\d z} \log \Gamma(z),\qquad \psi^{(m)}(z) = \frac{\d^m}{\d z^m} \psi(z) \\
	\psi^{(m)}(z) &= (-1)^{m+1}m! \sum_{k = 0}^\infty \frac{1}{(z+k)^{m+1}}, \qquad z \in \mathbb C \setminus \mathbb Z_{<0}, \ m \in \mathbb N^+.
\end{align*}
We have
\begin{equation}\label{eq:theta}
\begin{split}
	\theta(N) &= (N+1)\log d - \log(d-1) - \log N + \log r + \frac 1N \Big( \log \Big(\frac{e^r r}{\eps}\Big) - \log \Gamma(N+2) \Big) \\
	\theta'(N) &= \log d -\frac 1N -\frac{1}{N^2} \Big( \log \Big(\frac{e^r r}{\eps}\Big) - \log \Gamma(N+2) \Big) -\frac{\psi(N+2)}{N} \\
	\theta''(N) &= \frac{1}{N^2} +\frac{2}{N^3} \Big( \log \Big(\frac{e^r r}{\eps}\Big) - \log \Gamma(N+2) \Big) +\frac{2\psi(N+2)}{N^{2}} -\frac{\psi'(N+2)}{N}
\end{split}
\end{equation}
We have
\begin{align*}
	\eta(N) &\coloneqq N^3 \theta''(N)\big|_{\eps = e^r r} = N - 2\log \Gamma(N+2) + 2N\psi(N+2) - N^2\psi^{(1)}(N+2) \\
	\eta'(N) &= 1 - N^2\psi^{(2)}(N+2) = 1 + 2N^2 \sum_{k = 0}^\infty \frac{1}{(N+2+k)^3} > 0.
\end{align*}
Since $\eta(0) = 0$ we conclude $\eta > 0$ on $(0,+\infty)$ and the claim about $\phi$ follows, and since a convex function on $(0,+\infty)$ which diverges to $+\infty$ has a unique minimum we have that
\begin{equation}
	N^* \coloneqq \arg\min_{(0,\infty)} \phi, \qquad m^* \coloneqq m^*(N^*).
\end{equation}
is the unique solution to \eqref{eq:relaxed}.
\end{proof}
Plotting shows that the convexity claim fails (close to $0$) immediately once $\eps > e^r r$. Next, we view $N^*$ and $m^*$ as a functions of $r,\eps$ and show that both $N^*$ and $m^*$ diverge as $r \to +\infty$ or $\eps \to 0^+$. It will be convenient to introduce the parameter $\lambda$, which will occasionally be substituted for the function of $\eps, r$
\begin{equation}\label{eq:lambda}
\lambda(\eps, r) \coloneqq \log(r e^r/\eps).
\end{equation}
The following is the main result in this section, in which we obtain the asymptotics for $N^*$ and $m^*$ in terms of $r$ and $\eps$. The optimal truncation degree $N^*$ grows as the square root of $\log(1/\eps)$ and at the square root rate in $r$. The optimal number of intervals $m^*$ grows subpolynomially in $1/\eps$, but stretched-exponentially in $r$ (faster than every power of $r$ but slower than $e^{cr}$ for every $c>0$).
\begin{theorem}\label{thm:asympN*m*}
$N^*$ is strictly increasing as $r$ increases and as $\eps$ decreases, and
\[
N^* \sim \sqrt{\frac{\lambda(\eps, r)}{\log d}},\qquad m^* \sim r \sqrt{\frac{\log d}{\lambda(\eps, r)}} \exp \big(\sqrt{\lambda(\eps, r) \log d}\big)
\]
as $r \to +\infty$ or $\eps \to 0^+$, in each case with the other fixed.
\end{theorem}

We mention that the proof of the theorem above identifies the sharper approximate value of $N^*$ \eqref{eq:preciseN*} (necessary for the asymptotics of $m^*$), which may be of value for fixed choices of $r,\eps$.

\begin{proof}[Proof of \autoref{thm:asympN*m*}]
Since $N^*$ is the unique solution to $\phi'(N) = 0$, equivalent to $\theta'(N) = 0$, following \eqref{eq:theta} we consider the function
\begin{equation}\label{eq:alphaNlambda}
\alpha(N, \lambda) = \log d -\frac{1 + \psi(N+2)}{N} -\frac{\lambda - \log \Gamma(N+2)}{N^2}, \qquad N > 0, \ \lambda \geq 0
\end{equation}
so that $\alpha(N, \lambda(\eps,r)) = \theta'(N)$. Writing $N^* = N^*(\lambda)$ and differentiating implicitly
\[
\frac{\d N^*}{\d \lambda}(\lambda) = - \frac{\partial_\lambda \alpha(N^*(\lambda), \lambda)}{\partial_N \alpha(N^*(\lambda), \lambda)} = \frac{1}{N^*(\lambda)^2 \theta''(N^*(\lambda))} > 0.
\]
Since $\lambda(\eps,r)$ \eqref{eq:lambda} is strictly increasing in $r$ and strictly decreasing in $\eps$, the first claim follows.

The equation $\alpha(N, \lambda) = 0$ reads
\begin{equation}\label{eq:lambdaequals}
\lambda = N^2 \log d - N(1 + \psi(N+2)) + \log \Gamma(N+2).
\end{equation}
By Stirling's approximation $\log \Gamma(N+2)  \sim N\log N - N + \mathrm{O}(\log N)$ and by the series representation of the digamma function \cite[(6.3.18)]{handbook} $\psi(N+2) \sim \log N + \mathrm{O}(1/N)$. Substituting into \eqref{eq:lambdaequals} and viewing $N^* = N^*(\lambda)$ we obtain
\begin{equation}\label{eq:lambdaequals2}
\lambda = N^*(\lambda)^2 \log d - 2N^*(\lambda) + \mathrm{O}(\log N^*(\lambda))
\end{equation}
From this the leading asymptotics
\[
N^*(\lambda) \sim \sqrt{\frac{\lambda}{\log d}}
\]
follow, and substituting in $\lambda = \lambda(\eps,r)$ proves the claims regarding the asymptotics of $N^*$ as a function of $r$ and $\eps$. Now subtract $\lambda = \log d (\sqrt{\lambda/\log d})^2$ from both sides, and obtain
\[
\log d \bigg( (N^*)^2 - \frac{\lambda}{\log d} \bigg) = 2N^* + \mathrm{O}(\log N^*)
\]
Factoring the difference of squares, and using the fact that $N^* + \sqrt{\lambda/\log d} \sim 2N^*$ obtain
\[
N^* - \sqrt{\frac{\lambda}{\log d}} = \frac{1}{\log d} + \mathrm{O}\bigg(\frac{\log N^*}{N^*} \bigg)
\]
which can be rewritten as
\begin{equation}\label{eq:preciseN*}
N^* = \sqrt{\frac{\lambda}{\log d}} + \frac{1}{\log d} + \mathrm{o}(1)
\end{equation}
We turn to the asymptotics for $m^*$, which can be written as a function of lambda as
\[
m^*(N) = r \bigg( \frac{e^\lambda}{\Gamma(N+2)} \bigg)^{1/N} \implies \log m^*(N) = \log r + \frac{\lambda - \log \Gamma(N+2)}{N}.
\]
Since \eqref{eq:lambdaequals} can be rewritten as
\[
\frac{\lambda - \log \Gamma(N^*+2)}{N^*} = N^*\log d - 1 - \psi(N^*+2),
\]
matching terms, using \eqref{eq:preciseN*} and the asymptotics of the digamma function \cite[(6.3.18)]{handbook} $\psi(N+2) = \log N + \mathrm{o}(1)$ we obtain for $m^* = m^*(N^*)$
\begin{align*}
\log m^*
&= \log r + N^*\log d - 1 - \psi(N^*+2) \\
&= \log r + \sqrt{\lambda \log d} - \frac 12 \log \lambda + \frac 12 \log \log d + \mathrm{o}(1).
\end{align*}
Exponentiating yields the statement on the asymptotics for $m^*$.
\end{proof}

Finally, we analyse the asymptotics of the storage $f_d$ evaluated at a convenient choice of $N$ and $m$ (close to the optimum) and compare them to those of the naive solutions given by taking one of $N, m$ to be equal to $1$ and solving \eqref{eq:relaxed} for the other:
\begin{equation}\label{eq:naiveaxis}
N^\circ \coloneqq \min \bigg\{ N:  \frac{r^{N+1}e^r}{(N+1)!} \leq \eps \bigg\}, \qquad m^\circ \coloneqq \frac{r^2e^r}{2\eps}.
\end{equation}
We prove a result that should be interpreted as saying that there are choices of $(N, m)$ that are genuinely better than either saving the path as $(1, m^\circ)$. We do not show the corresponding statement for $(N^\circ, 1)$ but verify it experimentally in \autoref{fig:naive}; note that $N^\circ < \infty$ for any $r$ thanks to factorial decay.

\begin{proposition}\label{thm:optimal1var}
Recall the definition of the function $\lambda(\eps, r)$ \eqref{eq:lambda}. Let
\[
\widetilde N \coloneqq \sqrt{\frac{\lambda(\eps,r)}{\log d}}, \qquad \widetilde m \coloneqq m^*(\widetilde N).
\]
We have the bound
\[
f_d(\widetilde N, \widetilde m) \leq \frac{r d \log d}{(d-1)\lambda(\eps,r)} \exp\big(1 + 2 \sqrt{\lambda(\eps,r)\log d}\big) \in \mathrm{o}(f_d(1, m^\circ))
\]
as both $r \to +\infty$ and $\eps \to 0^+$ with the other fixed.
\end{proposition}
\begin{proof}
We have, using
\[
\log (N+1)! = \sum_{k = 1}^{N+1} \log k \geq \int_1^{N+1} \log x \d x \geq N \log N - N,
\]
the bound on $\theta$, using \eqref{eq:theta}
\begin{align*}
	\theta(N) &= (N+1)\log d - \log(d-1) - \log N + \log r + \frac{\lambda - \log \Gamma(N+2)}{N} \\
	&\leq (N+1)\log d - \log(d-1) - \log N + \log r + \frac{\lambda}{N} - \log N + 1.
\end{align*}
Substituting $N = \widetilde N$ we obtain
\begin{align*}
	\theta(\widetilde N) \leq 2 \sqrt{\lambda \log d} - \log \lambda + \log \Big( r \frac{d}{d-1} \log d  \Big) + 1
\end{align*}
and the claimed bound follows by exponentiating. Now consider 
\[
f_d(1, m^\circ) = \phi(1) =  \frac{r^2e^r}{2\eps}\frac{d^2}{(d-1)} = \frac{d^2 r e^\lambda}{2(d-1)}.
\]
On the other hand, we have just shown that
\[
f_d(\widetilde N, \widetilde m) \lesssim_d \frac{r}{\lambda} \exp\big( 2\sqrt{\lambda \log d} \big),
\]
and therefore
\[
\frac{f_d(\widetilde N, \widetilde m)}{f_d(1,m^\circ)} \lesssim_d \lambda^{-1} \exp \big( 2 \sqrt{\lambda \log d} - \lambda \big) \to 0
\]
since $\lambda \to \infty$ as either $r \to \infty$ and $\eps \to 0$.
\end{proof}

\begin{remark}\label{rem:allCDEs}
	We briefly comment on one of the reasons for our choice of considering linear CDEs and not more general non-linear ones $\d Y = V(Y)\d X$. Level-$N$ Euler estimates with the same decay in $m$ also exist for such equations (e.g.\ \cite[Theorem 10.30]{FV10} with $p = 1$). An important distinction from the linear case is that the factorial is not present in the error estimate \autoref{thm:euler1}, due to a combinatorial coefficients arising from the iteration of vector field composition. It is therefore not the case that the error decays as $N \to \infty$ for every $m$, although it does for $m$ sufficiently high (depending on $r$). The same occurs for the log-ODE method, also in the linear case, see \eqref{eq:logODE} below and related discussion. More practically, it would be much more difficult to sample choices of $V$ with $\|V\|_{\mathrm{Lip}^\gamma} = r$ in a way that is expected to be representative, compared to sample matrices with $|A| = r$.
\end{remark}

\begin{remark}
	The computational cost of exponentiating the log-signature in the Euler scheme may not be negligible for high-enough $N$; if this is considered unacceptable, one could store signatures instead of log-signatures, and modify the problem by replacing $\Lambda_d(N)$ with $\mathrm{dim}(T_N(\mathbb R^d))$ \eqref{eq:dimT}. This change does not affect the asymptotics derived in this section, but we observe experimentally \cite{code} that it may still have a small effect of decreasing $N^*$ and increasing $m^*$ for fixed choices of the parameters.
\end{remark}

In practice, it is possible to solve the original integer-valued minimisation problem of the exact storage cost $m\Lambda_d(N)$ subject to the Euler inequality $\leq \eps$ constraint, by computing for $N$ up to some $N_\mathrm{max}$ $\lceil m^*(N) \rceil\Lambda_d(N)$ (in fact the exact value of $m$ for each $N$ at which the constraint starts being true can be located with a small binary search) and safely stopping once $\Lambda_d(N)$ exceeds the optimum over $n < N$. The evaluation of $\Lambda_d(N)$, computed by summing Witt's formula \eqref{eq:witt} from $1$ to $N$ is instantaneous for all reasonable values of $d, N$. We refer to \autoref{fig:constrained}, \autoref{fig:one-var-sweeps} and \autoref{fig:naive} for some empirical validations of our results.

\begin{figure}[h!]
	\centering
	\includegraphics[width=0.6\textwidth]{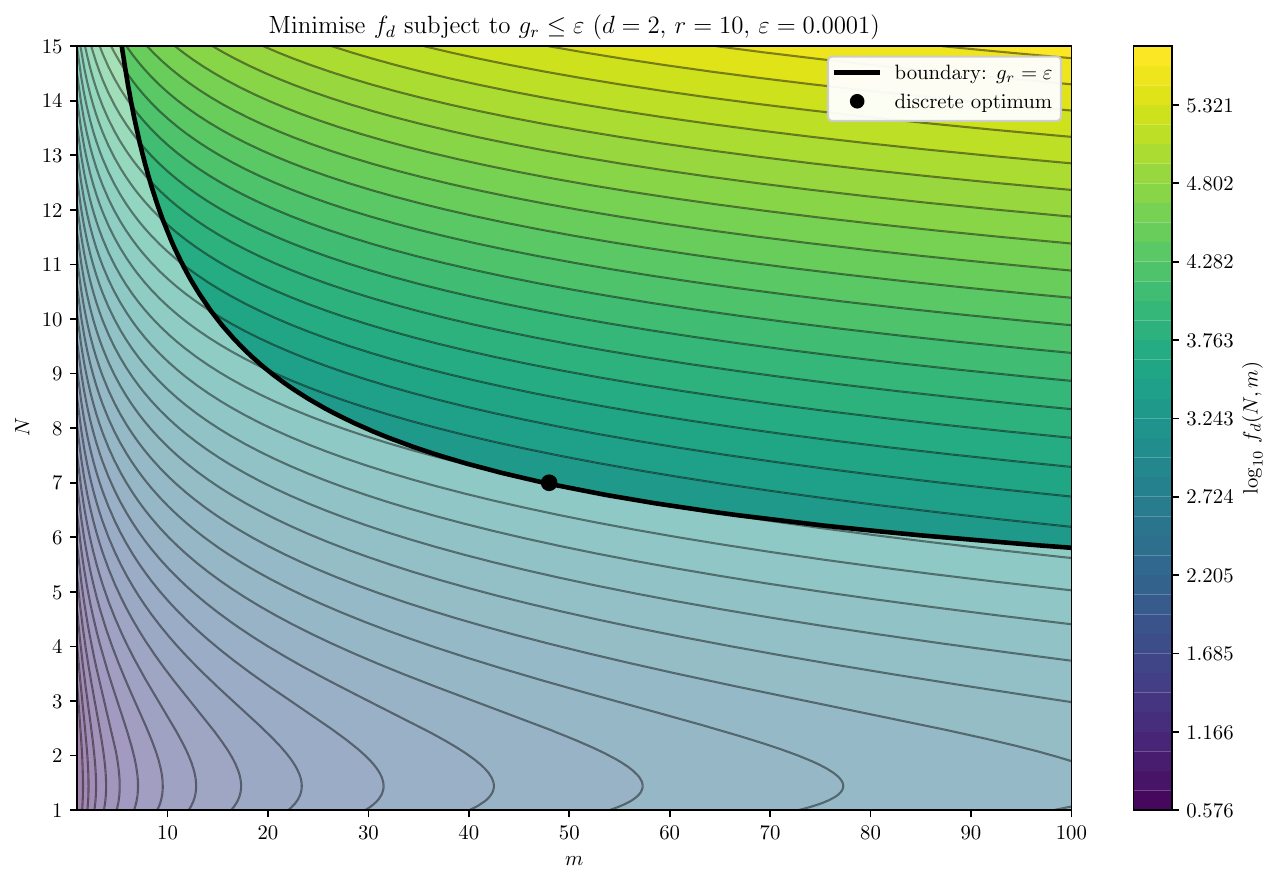}
	\caption{Plot of the constrained optimisation problem \eqref{eq:relaxed}. We also compute the discrete minimum of the actual storage cost $m\Lambda_d(N)$ (the black dot). As one would expect, this occurs at the point of tangency of the level sets of $f_d$ and the feasibility boundary of the constraint. The infeasible region $g_r > \eps$, closer to the axes, is shaded. One can verify, both in this and in other plots, that $f_d(N,m)$ is an excellent proxy for $m\Lambda_d(N)$ even non-asymptotically, and the resulting discrete problem most often has the same solution.}
	\label{fig:constrained}
\end{figure}

\begin{figure}[h!]
	\centering
	\setlength{\tabcolsep}{1pt}
	\begin{tabular}{@{}cc@{}}
		\begin{minipage}{0.4\textwidth}
			\centering
			\includegraphics[width=\textwidth]{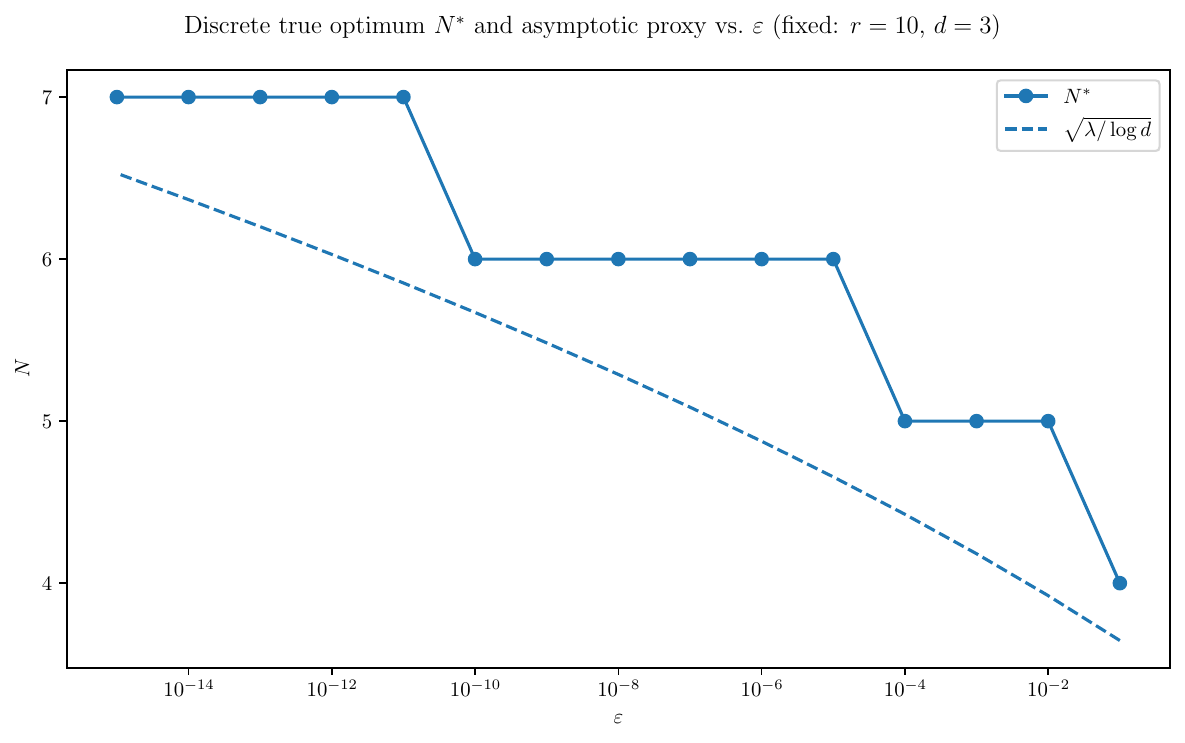}
		\end{minipage}
		&
		\begin{minipage}{0.4\textwidth}
			\centering
			\includegraphics[width=\textwidth]{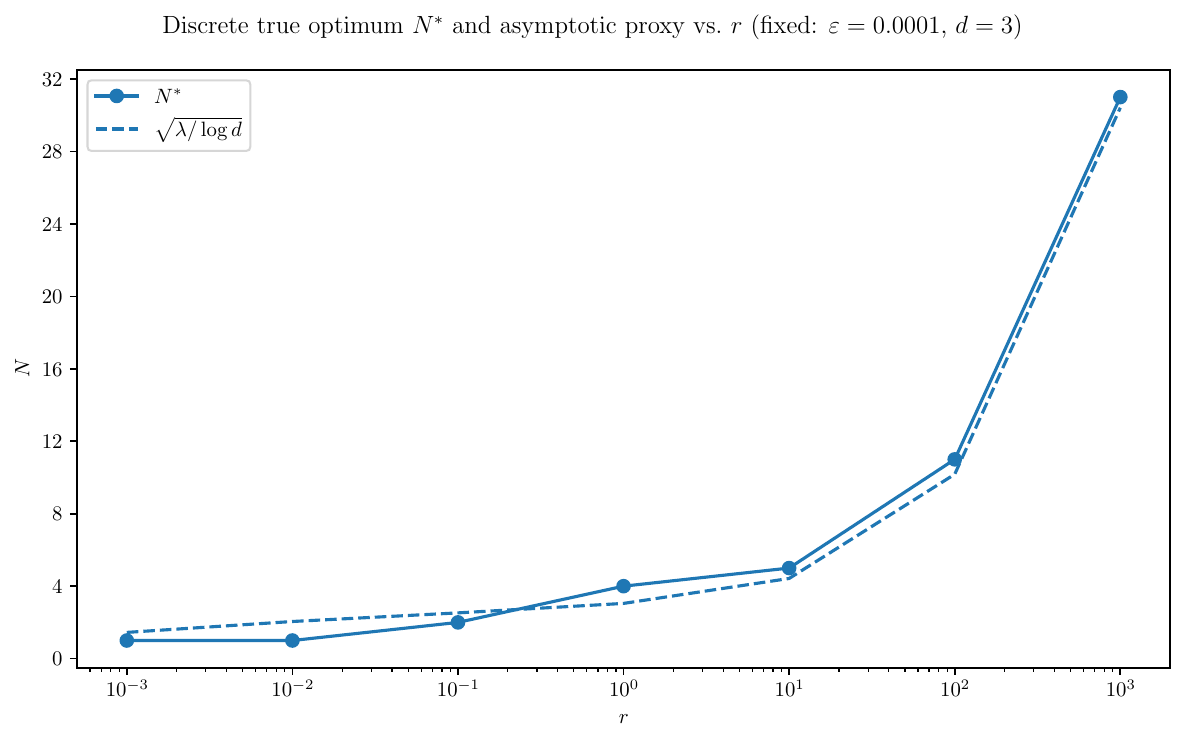}
		\end{minipage}
		\\[0.5em]
		\begin{minipage}{0.4\textwidth}
			\centering
			\includegraphics[width=\textwidth]{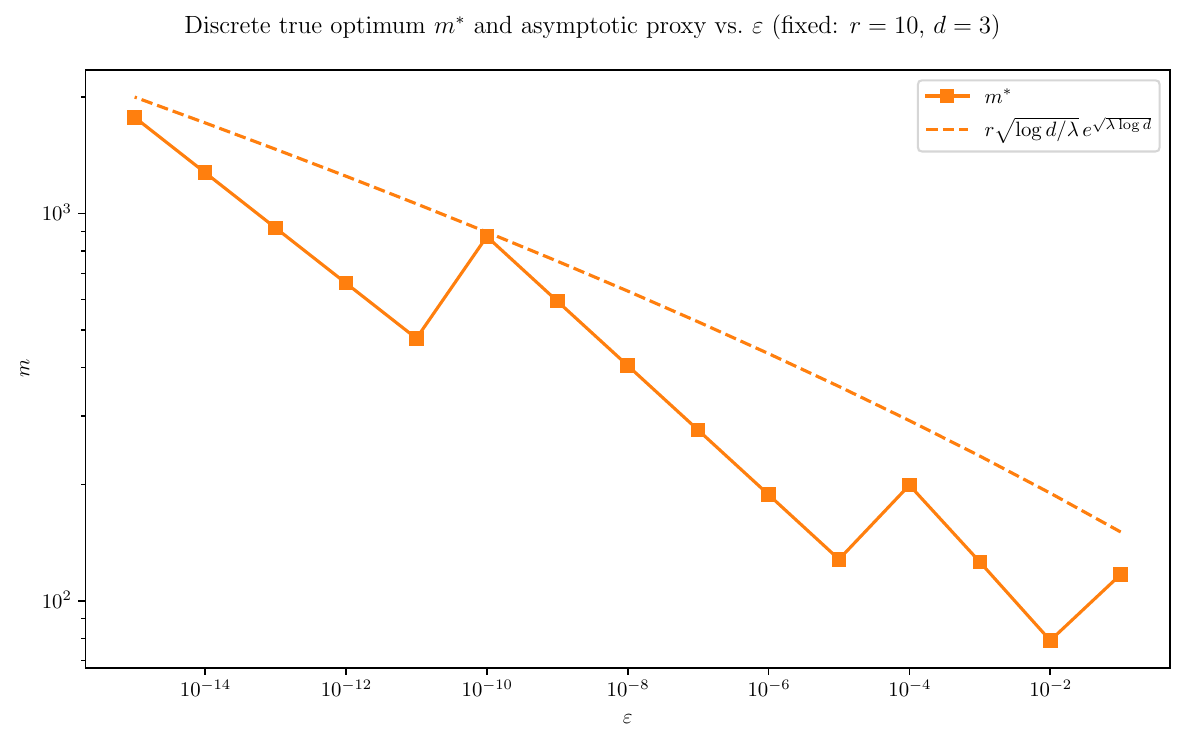}
		\end{minipage}
		&
		\begin{minipage}{0.4\textwidth}
			\centering
			\includegraphics[width=\textwidth]{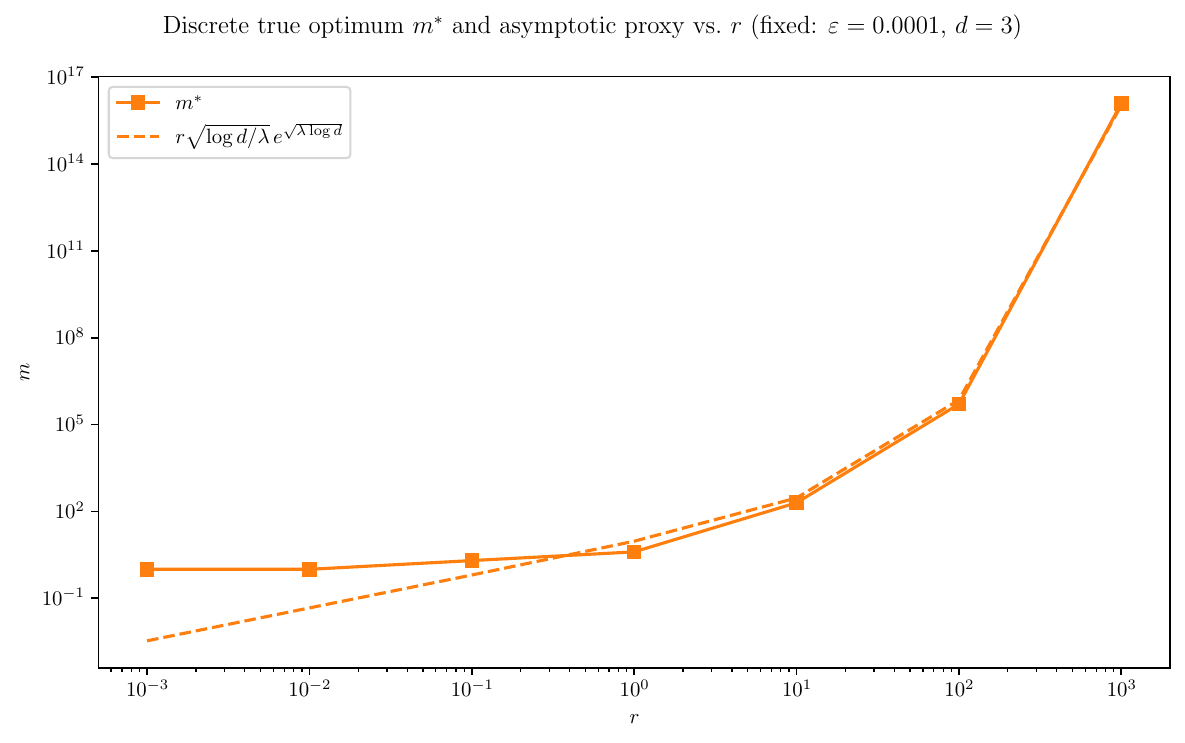}
		\end{minipage}
	\end{tabular}
	\caption{Comparison of the discrete optimum with true cost function $m\Lambda_d(N)$ with the asymptotics of \autoref{thm:asympN*m*}. Across most choices of the fixed parameters, the proxies track the true optima very well, even in the non-asymptotic regime.}
	\label{fig:one-var-sweeps}
\end{figure}

\begin{figure}[h!]
	\centering
	\begin{minipage}{0.4\textwidth}
		\centering
		\includegraphics[width=\textwidth]{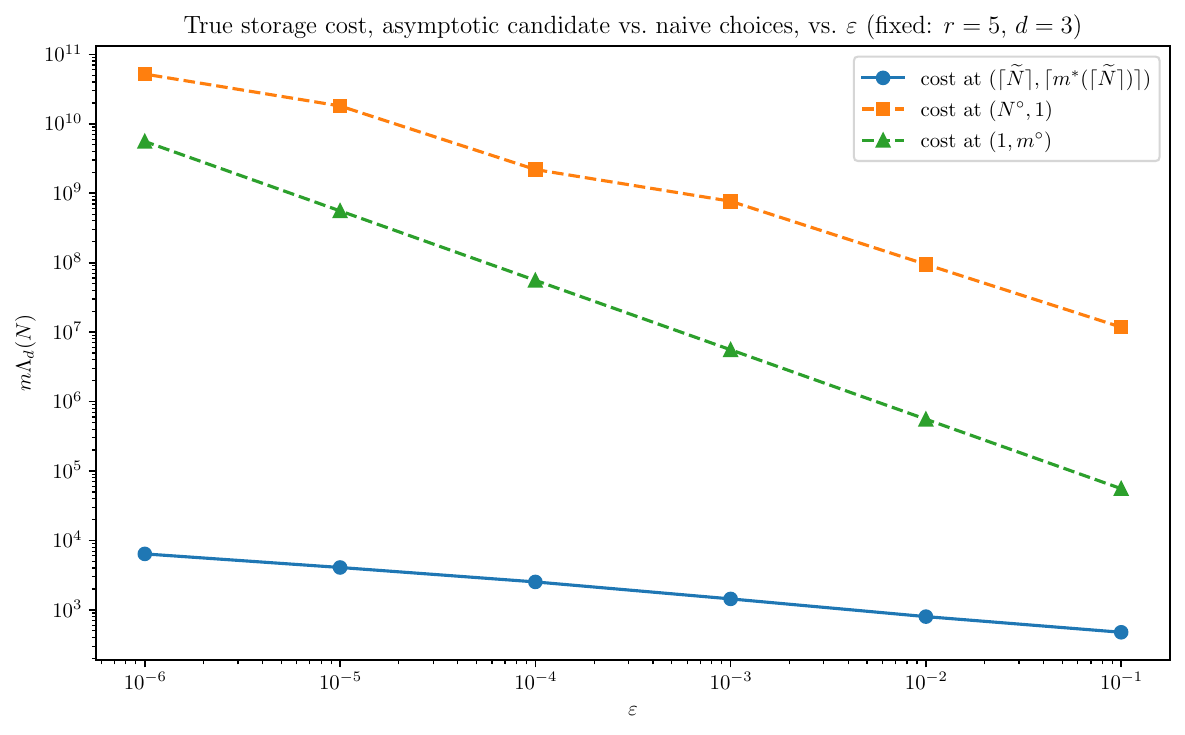}
	\end{minipage}
	\begin{minipage}{0.4\textwidth}
		\centering
		\includegraphics[width=\textwidth]{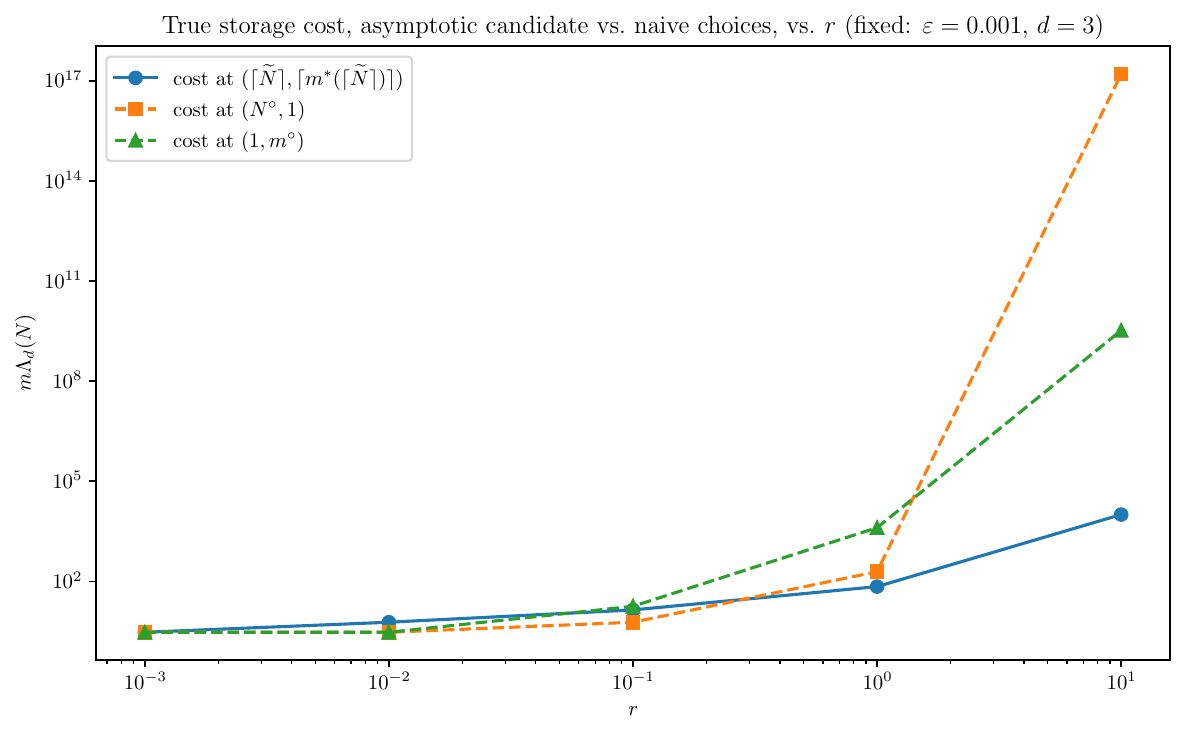}
	\end{minipage}
	
	\caption{Comparison of the true storage cost of the proxy $(\widetilde{N}, m^*(\widetilde{N}))$ for the optimum (or rather integer values of these which are guaranteed to be feasible, needed to compute the cost) with that of the naive choices \eqref{eq:naiveaxis}, as $\eps$ and $r$ vary, corroborating the ``strictness'' claim of \autoref{thm:optimal1var} and adding a similar one for the choice $(N^\circ, 1)$.}
	\label{fig:naive}
\end{figure}

\section{Paths of bounded $p$-variation}\label{sec:pvar}

We now relax the restriction on $X$ being of bounded variation to the requirement that it be of finite $p$-variation. In this case $X$ alone does not define solutions to differential equations. As shown by \cite{Lyo98}, the natural extension is a \emph{rough path} $\bX$. We refer to \cite{LCL, FV10} for background on rough path theory. We emphasise the following crucial fact, which will be explained in due course: the material in this section is also relevant for $C^\infty$ paths.

Let $p \in [1,\infty)$ and define the constant
\begin{equation}\label{eq:beta}
\beta_p \coloneqq \frac{p}{1 - 2^{1-\frac{\p + 1}{p}}} .
\end{equation}
Following \cite[p.240-241]{Lyo98}, consider the homogeneous norm on $\mathcal G_\p(\mathbb R^d)$ defined as follows, given the choice of a cross norm on tensor powers of $\bbR^d$:
\begin{equation}\label{eq:homnorm}
\vertiii{g} \coloneqq \vertiii{g}_p \coloneqq \max_{1 \leq i \leq \p} \big\{ \big(\beta_p (i/p)! |g_i| \big)^{1/i} \big\} .
\end{equation}
Let $\bX$ be a $p$-rough path defined on $[0,T]$, i.e. $\omega_p(0,T) < \infty$ with the control $\omega_p$ (recall the definition \cite[Definition 1.6]{FV10} of a control as a superadditive, continuous function which vanishes on the diagonal) defined by
\begin{equation}\label{eq:pvar}
\omega_p(s,t) \coloneqq \omega^{\bX}_p(s,t) \coloneqq \| \bX \|_{p\text{-}\mathrm{var}}^p \coloneqq \sup_{\pi \in \Pi[s,t]} \sum_{[u,v] \in \pi} \vertiii{\bX_{u,v}}^p.
\end{equation}
Then
\begin{equation}\label{eq:factless}
|\bX^n_{s,t}| \leq \frac{\vertiii{\bX_{s,t}}^n}{\beta_p (n/p)!} \leq \frac{\omega(s,t)^{n/p}}{\beta_p(n/p)!}, \qquad n = 1,\ldots,\p .
\end{equation}

The following lemma will be used to guarantee existence of a good choice of the partition.
\begin{lemma}\label{lem:pi}
	Given a control $\omega$ on $[0,T]$, for any $m$ there exists a partition with $m$ intervals $\pi = \{0 = t_0, t_1,\ldots, t_m = T\}$ such that $\omega(t_i,t_j) \leq \frac{j-i}{m} \omega(0,T)$ for all $i,j = 0,\ldots,m$ with $i < j$.
\end{lemma}
\begin{proof}
	We use the same choice as \cite[Lemma A.3]{maud2}, who argue inductively using the intermediate value theorem that there exists a partition such that $\omega(0,t_k) = \frac km \omega(0,T)$. Then by superadditivity it additionally holds that
	\[
	\omega(t_i,t_j) \leq \omega(t_0,t_j) - \omega(t_0,t_i) = \frac{j-i}{m} \omega(0,T) . \qedhere
	\]
\end{proof}

\begin{remark}\label{eq:optpart}
	For a time series, finding an approximately equal length split has linear complexity in the number of observations (one may compute the cumulative sums of $|X_{t_i}-X_{t_{i-1}}|$ and then locate the partition points), whereas for the $p$-variation control with $p>1$ the analogous problem is computationally harder, since even computing the ordinary $p$-variation $\| X \|_{\var{p}}$ requires an optimisation over partitions.
\end{remark}

We now state the factorial decay theorem for the signature of a $p$-rough path. We use the formulation of \cite[Theorem 2.3]{LX13}, which differs slightly from the original \cite[Theorem 2.2.1]{Lyo98}, namely in the following two aspects: (i) the sharpened version of the neo-classical inequality due to \cite{HH10}
\begin{equation}\label{eq:neoclassical}
\sum_{k = 0}^n \binom{n/p}{k/p} x^{k/p} y^{(n-k)/p} \leq p(x+y)^{n/p}	
\end{equation}
(which replaces $p^2$ with $p$) is used, and (ii) the proof performs diadic refinement (see e.g.\ \cite{FH20}) rather than coarsening the partition in a particular order, which gives rise to the form of the constant $\beta_p$ in \eqref{eq:beta}.

\begin{theorem}[Factorial decay of the signature, {\cite[Theorem 2.2.1]{Lyo98}}]\label{thm:factdec}
For a $p$-rough path $\bX$,
\[
|\Sig^n(\bX)_{s,t}| \leq \frac{\omega_p(s,t)^{n/p}}{\beta_p(n/p)!}
\]
for all $n \geq 1$, where $\beta_p$ is defined in \eqref{eq:beta} and $\omega_p$ in \eqref{eq:pvar}.
\end{theorem}

We continue to consider the linear differential equation \eqref{eq:linear} in the rough case; the series representation \eqref{eq:AyS} continues to hold \cite[Theorem 4.5]{LCL}, and we proceed with the same notations $\Phi, \mathcal E$, etc., with the goal of deriving a bound for the error produced by the Euler scheme on a partition. The proof of \autoref{thm:euler1} does not work in the $p$-variation case (even for $p \in (1,2)$), essentially for the same reason that factorial decay of the signature is more delicate for $p > 1$: indeed, attempting to use the Chen identity
\[
\Sig^n(\bX)_{0,T} = \lim_{m \to \infty}\sum_{\substack{n_1 + \ldots + n_m = n \\ n_k \leq \p}} \bX^{n_1}_{t_0,t_1} \otimes \cdots \otimes \bX^{n_m}_{t_{m-1},t_m}
\]
it is not possible to bound with \eqref{eq:factless} since one has to distinguish between the cases of $k$ for which $n_k = 0$; on the other hand, forgetting the factor $\beta_p$ in the denominator and the constraint $n_k \leq \p$ and using the $m$-fold neo-classical inequality results in a divergent factor of $p^{m-1}$. We will instead adapt the proof of the convergence of the Euler scheme for (non-linear) $\mathrm{Lip}^\gamma$ equations \cite[Theorem 10.30]{FV10}, which circumvents this difficulty by a telescoping argument. This result could be used directly by extending the linear vector fields to $\mathrm{Lip}^\gamma$ (even supported outside of the region given by the a priori bound of the linear RDE); however, doing so would give up factorial decay in the bound, a feature of the problem that we wish to preserve. The main difficulty in adapting the proof for the $\mathrm{Lip}^\gamma$ bound is that, in the linear case, it is necessary to have an a priori bound on $\cE_{N;\pi}$. This is achieved by bounding it by proximity to the true solution, for which an a priori bound \emph{is} available; this results in a recursion that can be resolved by an application of the discrete Gr\"onwall inequality. We begin with a few preliminary lemmas. Let
\begin{equation}
\exp_p(x) \coloneqq \sum_{n = 0}^\infty \frac{x^{n/p}}{(n/p)!} .
\end{equation}
denote the Mittag-Leffler function with parameter $1/p$, evaluated at $x^{1/p}$.
\begin{lemma}
For $x \geq 0$
\[
\sum_{k = n}^\infty \frac{x^{k/p}}{(k/p)!} \leq \frac{x^{n/p}}{(n/p)!}\exp_p(x) .
\]
\end{lemma}
\begin{proof}
This follows from the identity $x!y! \leq (x+y)!$, see \cite[Lemma 4]{maud1}. 
\end{proof}

\begin{lemma}
For $x,y \geq 0$
\[
\exp_p(x)\exp_p(y) \leq p\exp_p(x + y) .
\]
\end{lemma}
\begin{proof}
By the neo-classical inequality \cite{HH10}
\begin{align*}
\exp_p(x) \exp_p(y) ={}& \sum_{i = 0}^\infty \sum_{j = 0}^\infty \frac{x^{i/p}}{(i/p)!} \frac{x^{j/p}}{(j/p)!} \\
={}&\sum_{k = 0}^\infty \sum_{i = 0}^k \frac{x^{i/p}}{(i/p)!} \frac{y^{(k-i)/p}}{((k-i)/p)!} \\
\leq{}& p\sum_{k = 0}^\infty \frac{(x+y)^{k/p}}{(k/p)!} \\
={}&p\exp_p(x+y) . \qedhere
\end{align*}
\end{proof}

\begin{lemma}
Let $|A| \leq r$, then
\[
|\Phi_{s,t} z - \Phi_{s,t} y| \leq \exp_p(r^p\omega_p(s,t)) |z-y|
\]
\end{lemma}
\begin{proof}
By \autoref{thm:factdec}, dropping the $\beta_p \geq 1$ in the denominator (which is not present for $n = 0$) we have
\begin{align*}
|\Phi_{s,t} z - \Phi_{s,t} y| &\leq \sum_{n = 0}^\infty |A|^n |\Sig^n_{s,t}| |z-y| \leq \sum_{n = 0}^\infty r^n \frac{\omega_p(s,t)^{n/p}}{(n/p)!} |z-y| \leq \exp_p(r^p\omega_p(s,t)) |z-y| . \qedhere
\end{align*}
\end{proof}

\begin{lemma}\label{lem:MLbound}
For $x \geq 0$ and $p \geq 1$
\[
\exp_p(x) \leq pe^x, \quad \text{and} \quad \exp_p(x) \sim pe^x \quad \text{as } x \to \infty.
\]
\end{lemma}
\begin{proof}
This follows from contour integral formulae for the Mittag-Leffler function, see \cite[(4.7.21) Theorem 4.22]{GKMR20}, which for $\alpha = 1/p$ and $\beta = 1$ reduces to
\[
\exp_p(x) = pe^x - \frac{\sin(\pi/p)}{2\pi/p} \int_0^{+\infty} \frac{x e^{-r^p}}{r^2 - 2rx \cos(\pi/p) + x^2}\d r,
\]
and the integrand is easily checked to be non-negative for $x, r \geq 0$. The asymptotic statement can be found in \cite[Proposition 3.5]{GKMR20} (with their $p = 0$).
\end{proof}

\begin{theorem}[Euler error for linear RDEs]\label{thm:eulerp}
There exists a partition $\pi$ on $[0,T]$ with $m$ intervals (i.e.\ \autoref{lem:pi}) and $C_p(r, \omega_p(0,T); N, m)$ which converges to a constant $C_p$ (depending only on $p$) as either $N$ or $m \to \infty$, such that for $p < N + 1$
\begin{equation}\label{eq:eulerp'}
	|\Phi_{0,T} - \mathcal E_{N;\pi}| \leq C_p(r, \omega_p(0,T); N, m) \exp(2r^p \omega_p(0,T))  \frac{(r \omega_p(0,T)^{1/p})^{N+1}}{\big(\frac{N+1}{p}\big)! m^{\frac{N+1}{p} - 1} } 
\end{equation}
(The proof will provide the explicit values of $C_p(r, \omega_p(0,T); N, m)$ and $C_p$.)
\end{theorem}
The reason why it is not possible for us to state the bound with $C_p$ outright independent of the other parameters has to do with the difficulty of obtaining an independent an a priori bound for $\cE_{N;\pi}$. Modulo this rephrasing, \eqref{eq:eulerp'} has a similar form to the bound in the $1$-variation case \autoref{thm:euler1}, except for the fact that the exponential is squared, a consequence of the fact that the proof is different.
\begin{proof}[Proof of \autoref{thm:eulerp}]
We may omit the argument $y_0$ as in the proof of \autoref{thm:euler1}, and also simplify notation by omitting the argument $N$ in $\cE_{N,\pi}$ (and similar) by writing $\Sig^n_{s,t} \coloneqq \Sig^n(X)_{s,t}$ and $\omega \coloneqq \omega_p$. For $l = 1,\ldots,m$
\begin{align*}
&|\Phi_{t_0,t_l} - \mathcal E_{\pi[t_0,t_l]}| \\
\leq{} &\sum_{k = 1}^l |\Phi_{t_k, t_l}   \mathcal E_{\pi[t_0, t_k]} - \Phi_{t_{k-1}, t_l}   \mathcal E_{\pi[t_0,t_{k-1}]}| \\
={}&\sum_{k = 1}^l |\Phi_{t_k, t_l}   \mathcal E_{\pi[t_0, t_k]} - \Phi_{t_{k}, t_l}   \Phi_{t_{k-1}, t_k}   \mathcal E_{\pi[t_0,t_{k-1}]}| \\
\leq{}& \sum_{k = 1}^l \exp_p(r^p \omega(t_k,t_l)) |\mathcal E_{\pi[t_0, t_k]} - \Phi_{t_{k-1}, t_k}   \mathcal E_{\pi[t_0,t_{k-1}]}| \\
={}& \sum_{k = 1}^l \exp_p(r^p \omega(t_k,t_l)) \Big|\sum_{n = 0}^N A^{\otimes n} \mathcal E_{\pi[t_0,t_{k-1}]} \Sig^n_{t_{k-1},t_k} - \sum_{n = 0}^\infty A^{\otimes n} \mathcal E_{\pi[t_0,t_{k-1}]} \Sig^n_{t_{k-1},t_k} \Big| \\
\leq{}& \sum_{k = 1}^l \exp_p(r^p\omega(t_k,t_l)) \sum_{n = N+1}^\infty r^n|\mathcal E_{\pi[t_0, t_{k-1}]}||\Sig^n_{t_{k-1},t_k}| \\
\leq{}& \sum_{k = 1}^l \exp_p(r^p\omega(t_k,t_l)) |\mathcal E_{\pi[t_0, t_{k-1}]}| \sum_{n = N+1}^\infty r^n \frac{\omega(t_{k-1},t_k)^{n/p}}{(n/p)!} \\
\leq{}& \sum_{k = 1}^l \exp_p(r^p\omega(t_k,t_l)) |\mathcal E_{\pi[t_0, t_{k-1}]}| \frac{r^{N+1}\omega(t_{k - 1},t_k)^{\frac{N+1}{p}}}{\beta_p\big(\frac{N+1}{p}\big)!} \exp_p(r^p\omega(t_{k-1},t_k)) \\
\leq{}&\frac{p r^{N+1}}{\beta_p\big(\frac{N+1}{p}\big)!} \sum_{k = 1}^l |\mathcal E_{\pi[t_0, t_{k-1}]}| \exp_p(r^p \omega(t_{k-1},t_l)) \omega(t_{k-1},t_k)^{\frac{N+1}{p}} \\
\leq{}&\frac{p r^{N+1}}{\beta_p\big(\frac{N+1}{p}\big)!} \sum_{k = 1}^l (|\Phi_{t_0,t_{k-1}} - \mathcal E_{\pi[t_0,t_{k-1}]}| + |\Phi_{t_0,t_{k-1}}| ) \exp_p(r^p \omega(t_{k-1},t_l)) \omega(t_{k-1},t_k)^{\frac{N+1}{p}} \\
\leq{}&\frac{p r^{N+1}}{\beta_p\big(\frac{N+1}{p}\big)!} \sum_{k = 1}^l |\Phi_{t_0,t_{k-1}} - \mathcal E_{\pi[t_0,t_{k-1}]}| \exp_p(r^p \omega(t_{k-1},t_l)) \omega(t_{k-1},t_k)^{\frac{N+1}{p}} \\
&+ \frac{p^2 r^{N+1}}{\beta_p\big(\frac{N+1}{p}\big)!} \sum_{k = 1}^l \exp_p(r^p \omega(t_0,t_l)) \omega(t_{k-1},t_k)^{\frac{N+1}{p}} \\
\leq{} &\frac{p r^{N+1}}{\beta_p\big(\frac{N+1}{p}\big)!} \omega(0,T)^{\frac{N+1}{p}} m^{-\frac{N+1}{p}} \sum_{k = 1}^l |\Phi_{t_0,t_{k-1}} - \mathcal E_{\pi[t_0,t_{k-1}]}| \exp_p(r^p \tfrac{l-k+1}{m} \omega(0,T))  \\
&+ \frac{p^2 r^{N+1} }{\beta_p\big(\frac{N+1}{p}\big)!} m^{1-\frac{N+1}{p}} \exp_p(r^p \omega(0,T)) \omega(0,T)^{\frac{N+1}{p}}.
\end{align*}
Recall the following form of the discrete Gr\"onwall inequality \cite{clark87} (and $1+b \leq e^b$): for $a_l,b_l,\varepsilon_l \geq 0$
\begin{equation}\label{eq:discgron}
\varepsilon_l \leq a_l + \sum_{h = 0}^{l-1} b_h \varepsilon_h \quad \implies \quad \varepsilon_l \leq (\max_{h = 0,\ldots,l} a_h) \exp\Big(\sum_{h = 0}^{l-1} b_h \Big) . 
\end{equation}
Calling $\varepsilon_l = |\Phi_{0,t_l} - \mathcal E_{\pi[t_0,t_l]}|$ with $h = k-1$, the bound obtained is of the form
\[
\varepsilon_l \leq a + b\sum_{h = 0}^{l-1} \exp_p(c (l - h)) \varepsilon_h
\]
(with $a,b,c$ dependent on $m$ but not $l$). Performing the transformation $\widetilde \varepsilon_l = \varepsilon_l \exp_p(cl)^{-1}$, and observing that
\[
\exp_p(c(l - h)) \exp_p(ch) \leq p \exp_p(cl) \ \implies \ \exp_p(c(l - h)) \exp_p(cl)^{-1} \leq p \exp_p(ch)^{-1} 
\]
we obtain
\begin{align*}
\widetilde \varepsilon_l &\leq a \exp_p(cl)^{-1} + b\sum_{h = 0}^{l-1} \exp_p(c (l - h))\exp_p(cl)^{-1} \varepsilon_h \leq a + bp\sum_{h = 0}^{l-1} \widetilde\varepsilon_h \\
\implies \ \widetilde\varepsilon_l &\leq a e^{bpl}\\
\implies \ \varepsilon_l &\leq a \exp_p(cl) e^{bpl} \leq a \exp_p(cm) e^{bpm} .
\end{align*}
From this it follows that for $l = 0,\ldots, m$
\begin{align*}
&|\Phi_{t_0,t_l} - \mathcal E_{\pi[t_0,t_l]}|\\
\leq{}& \frac{p^2 r^{N+1}}{\beta_p} \exp_p(r^p \omega(0,T))^2 \omega(0,T)^{\frac{N+1}{p}} \exp\bigg( \frac{p^2 r^{N+1}}{\beta_p\big(\frac{N+1}{p}\big)!} \omega(0,T)^{\frac{N+1}{p}} m^{1-\frac{N+1}{p}} \bigg) \frac{m^{1-\frac{N+1}{p}}}{\big(\frac{N+1}{p}\big)!} .
\end{align*}
from which the conclusion follows after applying \autoref{lem:MLbound}.
\end{proof}

In theory, the problem for $p$-rough paths is amenable to a similar kind of treatment as the $1$-variation case, and asymptotics on the resulting values of $N^*$, $m^*$ can be derived which are qualitatively similar to those in the $1$-variation case. We find, however, that this approach loses some of its importance for three main reasons: first, the bound involves constants and we do not expect it to be asymptotically sharp as the $1$-variation bound is, second the aforementioned \eqref{eq:optpart} difficulty in finding a good choice of $\pi$, and third, the fact that the optimisation should be stated over all $p \leq N$. Indeed, the optimisation problem should be stated as something like
\begin{align*}
	\text{minimise}\quad &m\frac{d^{N+1}}{(d-1)N} \quad \text{over } N, m \in (0,+\infty) \\
\text{subject to}\quad &\inf_{p < N + 1}  \bigg[ C_p \exp(2r^p \omega_p(0,T))  \frac{(r \omega_p(0,T)^{1/p})^{N+1}}{\big(\frac{N+1}{p}\big)! m^{\frac{N+1}{p} - 1} } \bigg] \leq \eps.
\end{align*}
where $\omega_p$ is the $p$-variation of the rough path $\Sig_{\p}(\bX)$; if $\bX$ itself is a $p_0$-rough path, the infimum should be taken over $p_0 \leq p < N + 1$. The point here is that taking higher $N$ unlocks the possibility of estimating the error in terms of $\omega_p(0,T)$ with higher $p$, which provide a much better estimate even if $X$ is smooth (but has very large $1$-variation norm). Unfortunately, this means that in order to approach this problem one must not only be able to compute the rough $p$-variation for a single value of $p$, but for the whole \emph{$p$-variation curve} of $X$ (at least for integer values of $p$)
\begin{equation}
[1,+\infty) \ni p \mapsto \omega_p^{\Sig_\p(X)}(0,T),
\end{equation}
i.e.\ the $p$-variation of the degree-$\p$ Stieltjes lift of $X$, an infinite-dimensional object. A much more practical approach is instead to test various choices of $(N,m)$ by randomly sampling matrices of norm $r$ and measuring the Euler error empirically for paths of interest. We do this in \autoref{subsec:pure} below for a toy example of smooth path, and in \autoref{sec:random} below in the stochastic setting.\\

A perspective that is slightly different compared to the Euler scheme is offered by the log-ODE method (see \cite{BGLY14} for the rough path formulation), which for linear equations is very closely related to the Magnus expansion \cite{magnus}. Define
\begin{equation}\label{eq:logODE}
	\frac{\d}{\d u}\cM_{N;s,t}^A(\bX)_u = \sum_{n = 1}^N A^{\otimes n} \cM_{N;s,t}^A(\bX)_u \pi^n\log \Sig(\bX)_{s,t}, \qquad 
\end{equation}
We write $\cM_N^A(\bX)_\pi$ to be the concatenation of the paths $\cM_{N;t_{i-1},t_i}^A(\bX)$, equivalently the solution to the ODE with piecewise-defined vector field. This coincides with the RDE (given by the original $A$) driven by the log-linear $N$-rough path $\bX_{N;\pi}$ \cite{flint}
\[
(\bX_{N;\pi})_{u,v} \coloneqq \pi_N \exp \bigg( \frac{v-u}{t_i - t_{i-1}} \pi_N \log \Sig(\bX)_{t_{i-1},t_i} \bigg)
\]
for $[u,v] \subset [t_{i-1},t_i]$ and extended on arbitrary intervals by requiring the Chen identity to hold. The log-ODE method has the advantage over the Euler scheme that the approximation also solves a differential equation and is therefore more easily interpreted geometrically. On the other hand, the Euler scheme has the computational advantage of being implementable in finitely many steps without requiring calls to an ODE solver. More importantly, for the asymptotics considered in this paper, while the log-ODE method has the same rate of $\smash{m^{1 - \frac{N+1}{p}}}$ in the number of (appropriately spaced) intervals, it fails to have factorial decay in the truncation level in general, similarly to how the Euler scheme for non-linear RDEs \autoref{rem:allCDEs} (see \autoref{fig:logode}). The obstruction can be obtained by considering the RDE for $\Sig_{N+1}(\bX)_{0,T}$ itself, and the degree-$N$ log-ODE approximation to it on a single interval $[0,T]$: $\mathcal M_{N;0,T} = \pi_{N+1} \exp (\pi_N \log \Sig_{0,T})$ the error is given by 
\[
\Sig_{N+1;0,T} - \mathcal M_{N;0,T} = \pi^{N+1} \log \Sig_{0,T},
\]
whose norm cannot decay factorially for all paths satisfying the Lyons-Sidorova conjecture \cite{LS06} (the fact that the differential equation depends on $N$ does not alter this statement, since its norm grows only geometrically). Compare with the step-$N$ Euler estimate, $\cE_{N;0,T} = \Sig_{N;0,T}$, which results in the error $\Sig^{N+1}_{0,T}$ that does decay factorially. Nevertheless, the log-ODE method permits the following more elegant alternative formulation of the problem.

\begin{definition}[$(\varepsilon, r)$-close rough paths]
	Let $\bX^1$ and $\bX^2$ be rough paths defined on $[0,T]$ (possibly of different regularity/degree). We say that they are $(\varepsilon, r)$-close if for all $A$ with $|A| \leq r$ it holds that $|\Phi^A_{0,T}(\bX^1) - \Phi^A_{0,T}(\bX^2)| < \varepsilon$.
\end{definition}

If we choose to approximate the solution using the log-ODE method instead of with the Euler scheme, the problem can be restated in terms of finding $N$ and $m$ (as well as $\pi$ of size $m$ satisfying \autoref{lem:pi}) such that $\bX_{N;\pi}$ is $(\eps,r)$-close to $\bX$.

\begin{figure}[h!]
	\centering
	\includegraphics[width=0.5\textwidth]{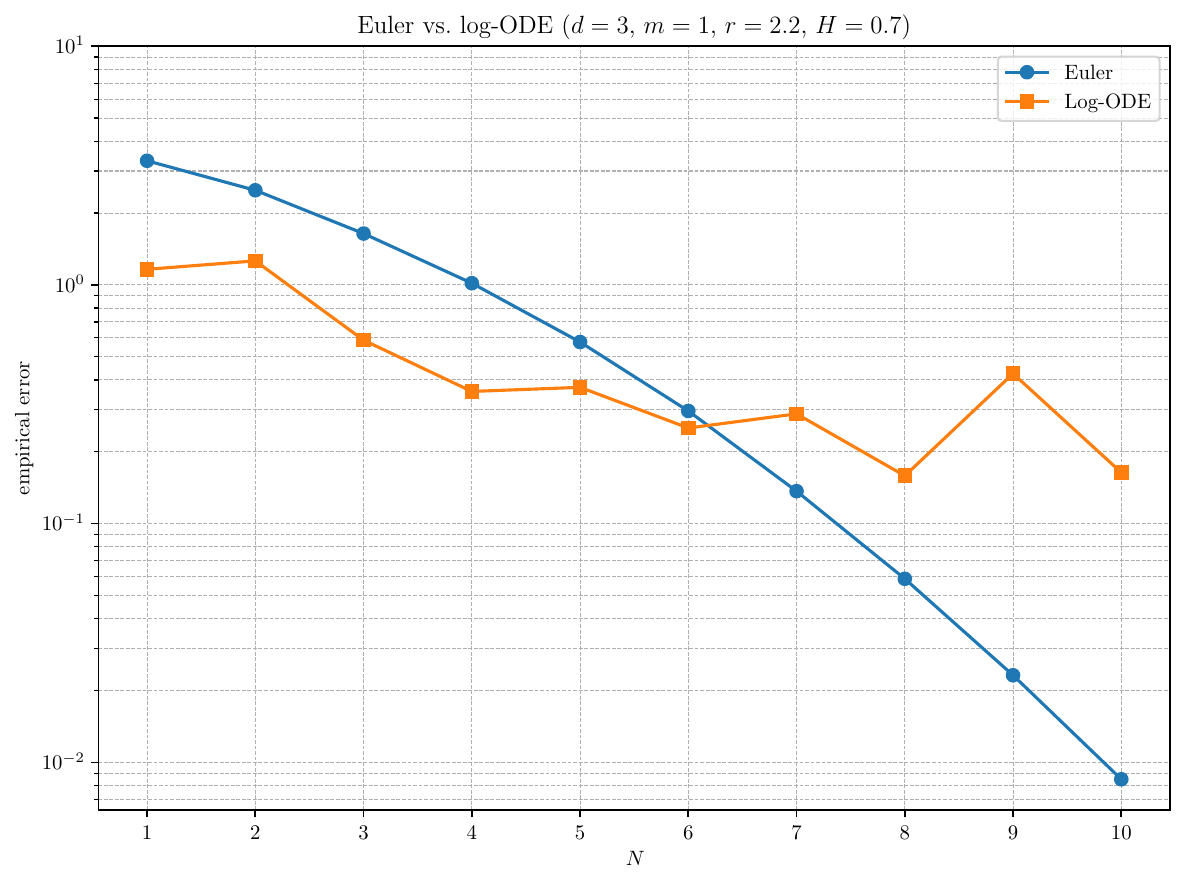}
	\caption{The one-step ($m = 1$) Euler scheme vs.\ log-ODE method. We have computed the $L^2$ norm of the error over $32$ paths sampled from fractional Brownian motion with Hurst paramter $H = 0.7$, averaged over $10$ random matrices with norm $r = 2.2$. Here and in all plots below, errors are compared against the solution computed using \texttt{Diffrax} \cite{kidger_diffrax, kidger2021on} on a very fine partition. For high enough $r = |A|$, the latter can fail to converge with $m$ fixed and $N \to \infty$, although for $r$ too large convergence of the Euler scheme in this sense can be initially slow. In fact, for low $N$, the log-ODE error appears to be lower than the Euler error, a finding which is superficially in agreement with the asymptotic efficiency result of \cite{CG95}.}
	\label{fig:logode}
\end{figure}

\subsection{Approximations to pure level-$\nu$ rough paths.}\label{subsec:pure}

In this subsection we focus on a class of \say{toy examples} of smooth paths which converge in $p$- but not $q$-rough path metric for $q < \p$. Let $\nu$ be a positive integer and $X \in C^{\var{1}}([0,1], \mathbb R^d)$ be such that $\Sig^n(X)_{0,1} = 0$ for $n = 1,\ldots,\nu - 1$, $\Sig^\nu(X)_{0,1} = \ell$ for some $0 \neq \ell \in (\mathbb R^d)^{\otimes \nu}$. Such paths have been constructed in the form of planar sinusoidal curves \cite{GK14}, see also \cite[\S 1.6]{Reiz19}. For $\nu = 2$ any closed curve suffices, for $\nu = 3$ an example is the \say{figure 8 curve} \autoref{fig:your-label}. Define, for a positive integer $\varrho$,
\begin{equation}\label{eq:approxPure}
X^{(\varrho)} \in C^{\var{1}}([0,1], \mathbb R^d),\qquad X^{(\varrho)}_\cdot \coloneqq \varrho^{-1}(X_{\varrho^\nu \cdot})^{\star \varrho^{\nu}},
\end{equation}
where $\star$ denotes concatenation (modulo translation so that endpoints coincide, for the trivial case $\nu = 1$). Notice that $\ell \in \mathfrak L^\nu(\bbR^d)$ since $\log \Sig(X)_{0,1} = \ell +$ higher-order terms, and since $\Sig(X)_{0,1} = \sum_{n \geq 0} \frac{1}{n!} (\log \Sig(X)_{0,1})^{\otimes n}$ the next term in $\Sig(X)_{0,1}$ has degree $2\nu$.

We now show that, in the limit of $\varrho \to \infty$, $\Sig_\nu(X^{(\varrho)})$ has the meaningful interpretation as a \emph{pure level-$\nu$ rough path}. Indeed, let $\bX$ denote the $\nu$-rough path $\bX_{s,t} = (t-s)\ell$. This is a special case of the pure rough paths introduced in \cite{BGS20} in which $\ell$ is a homogeneous Lie polynomial of degree $\nu$, and \cite[Proposition 2.11]{BGS20} implies
\begin{equation}\label{eq:skipdeg}
\Sig(X^{(\varrho)})_{s,t} \to \Sig(\bX)_{s,t} = \exp((t-s)\ell) \in \bigoplus_{k = 0}^\infty (\mathbb R^d)^{\otimes \nu k} .
\end{equation}
For $\nu = 1$ this is just a line segment and for $\nu = 2$ we recover the foundational example of pure area rough path \cite[Example 1.1.1]{Lyo98}. To show convergence $\Sig_{\nu}(X^{(\varrho)}) \to \bX$ in $(\nu, \nu + 1) \ni p$-rough path topology it suffices, by interpolation \cite[Lemma 8.16]{FV10} to establish uniform convergence $\Sig_{\nu}(X^{(\varrho)})_{s,t} \to \bX_{s,t}$ a uniform bound on $\| \Sig_\nu(X^{(\varrho)}) \|_{\var{p}}$ (cf.\ \cite[Exercise 2.10]{FH20}).

For $t-s \leq \varrho^{-\nu}$ and any $n \geq 1$ we have the bound
\begin{equation}\label{eq:tssmall}
\begin{split}
|\Sig^n(X^{(\varrho)})_{s,t}| &= \bigg| \int_{s < u_1 < \ldots < u_n < t} \dot X_{u_1}^{(\varrho)} \otimes \cdots \otimes \dot X_{u_n}^{(\varrho)} \d u_1 \cdots \d u_n \bigg| \\
&\leq \frac{\lVert \dot X^{(\varrho)} \rVert_{\infty}^n(t-s)^n}{n!} \\
&= \frac {\lVert \dot X \rVert_{\infty}^n}{n!} (t-s)^n \varrho^{n(\nu - 1)}.
\end{split}
\end{equation}
For $s \in [i\varrho^{-\nu}, (i+1)\varrho^{-\nu})$ and $t \in (j\varrho^{-\nu}, (j+1)\varrho^{-\nu}]$, $i,j \in \{0,\ldots,\varrho^\nu - 1\}$ and $k = j - i \geq 1$ we compute
\begin{align*}
&\Sig^n(X^{(\varrho)})_{s,t} \\
={} &\sum_{n_0 + \ldots + n_k = n} \Sig^{n_0}(X^{(\varrho)})_{s,(i+1)\varrho^{-\nu}} \Sig^{n_1}(X^{(\varrho)})_{(i+1)\varrho^{-\nu}, (i+2)\varrho^{-\nu}} \cdots \Sig^{n_{k-1}}(X^{(\varrho)})_{(j-1)\varrho^{-\nu}, j\varrho^{-\nu}} \Sig^{n_k}(X^{(\varrho)})_{j\varrho^{-\nu}, t} \\
={} &\varrho^{-n} \sum_{n_0 + \ldots + n_k = n} \Sig^{n_0}(X)_{s\varrho^\nu - i,1} \Sig^{n_1}(X)_{0,1} \cdots \Sig^{n_{k-1}}(X)_{0,1} \Sig^{n_k}(X)_{0, t\varrho^\nu - j} .
\end{align*}
If $1 \leq n < \nu$, bounding as in \eqref{eq:tssmall}
\begin{equation}\label{eq:nlessnu}
\begin{split}
|\Sig^n(X^{(\varrho)})_{s,t}| &\leq \varrho^{-n} \sum_{n_0 + n_{k+1} = n} |\Sig^{n_0}(X)_{s\varrho^\nu - i,1}| |\Sig^{n_{k+1}}(X)_{0, t\varrho^\nu - j}| \\
&\leq \varrho^{-n} \sum_{i+j = n} \frac{\|\dot X\|_\infty^i \|\dot X\|_\infty^j}{i!j!}\\
&= \frac{(2\|\dot X\|_\infty)^n}{n!} \varrho^{-n}
\end{split}
\end{equation}
At degree $\nu$ and $s,t$ in the same regime we have
\begin{equation}\label{eq:atnu}
|\Sig^\nu(X^{(\varrho)})_{s,t}| \sim \varrho^{-\nu}k \ell + \mathrm{O}(\varrho^{-\nu}) \sim (t-s)\ell + \mathrm{O}(\varrho^{-\nu})
\end{equation}
with the leading term coming from the sum $\varrho^{-\nu}\sum_{h  = 1}^k \Sig^\nu(X)_{0,1}$ and the remainder similar to the case above.

We have, following the definition \eqref{eq:pvar}
\begin{align*}
\omega_p^{\Sig_\nu(X)}(0,1) \asymp \sup_{\pi \in \Pi[0,1]} \sum_{[s,t] \in \pi} \max_{1 \leq n \leq \nu} \big\{ |\Sig^n(X^{(\varrho)})_{s,t}|^{p/n} \big\}
\end{align*}
For $p \in [\nu, \nu + 1)$ and by using \eqref{eq:tssmall} in the regime $t-s \leq \varrho^{-\nu}$ and \eqref{eq:nlessnu}, \eqref{eq:atnu} in the regime $t-s \geq \varrho^{-\nu}$
\begin{align*}
\omega_p^{\Sig_\nu(X)}(0,1) &\lesssim \sup_{\pi \in \Pi[0,1]} \sum_{[s,t] \in \pi} [\varrho^{p(\nu - 1)}(t-s)^p \mathbbm 1_{t-s \leq \varrho^{-\nu}}] \vee [(\varrho^{-p}  \vee (t-s) )\mathbbm 1_{t-s \geq \varrho^{-\nu}}] \\
&\lesssim \varrho^\nu \varrho^{p(\nu - 1)} (\varrho^{-\nu})^p \vee \varrho^{\nu - p} \vee 1 \\
&\asymp \varrho^{\nu - p} \vee 1 \\
&< \infty
\end{align*}
where in the first case we have chosen the maximising partition for which $t - s \equiv \varrho^{-\nu}$ with $\varrho^\nu$ intervals and in the second case we have similarly used that there can be at most $\varrho^\nu$ intervals with $t-s \geq \varrho^{-\nu}$. 

Conversely, let $p < \nu$. Since $\ell \neq 0$ there exists $\tau \in (0,1)$ such that 
\[
\max_{1 \leq n \leq \p} |\Sig^n(X)_{0,\tau}| \vee |\Sig^n(X)_{\tau,1}| > 0.
\]
Taking the partition $\pi$ with intervals of length alternating between $\varrho^{-\nu}\tau$ and $\varrho^{-\nu}(1 - \tau)$, we have
\[
\omega_p^{\Sig_\nu(X)}(0,1) \gtrsim \varrho^{\nu-p} \to \infty.
\]
We have therefore shown the following:
\begin{proposition}
$\Sig_\nu(X) \to \bX$ in $(\nu, \nu + 1) \ni p$-rough path metric, while for $p < \nu$, $\omega_p^{\Sig_\nu(X)}(0,1)$ diverges at rate $\varrho^{\nu - p}$.
\end{proposition}
In particular, for the bound of \autoref{thm:eulerp} to continue to be meaningful with $\varrho$ large, $m$ must be chosen exponentially in (a positive power of) $\varrho$.

\begin{figure}[h!]
	\centering
	\begin{minipage}[t]{0.255\textwidth}
		\centering
		\includegraphics[width=\linewidth]{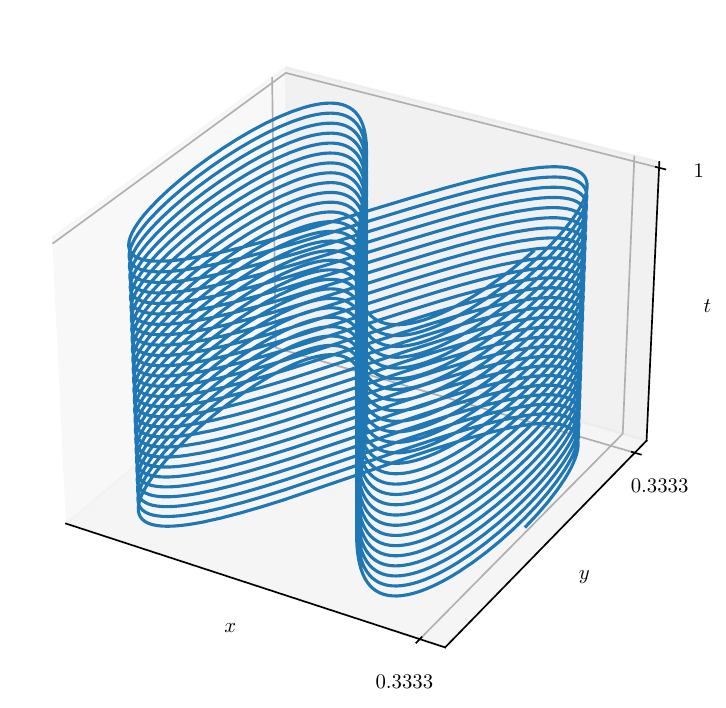}
	\end{minipage}\hspace{0.006\textwidth}%
	\begin{minipage}[t]{0.352\textwidth}
		\centering
		\includegraphics[width=\linewidth]{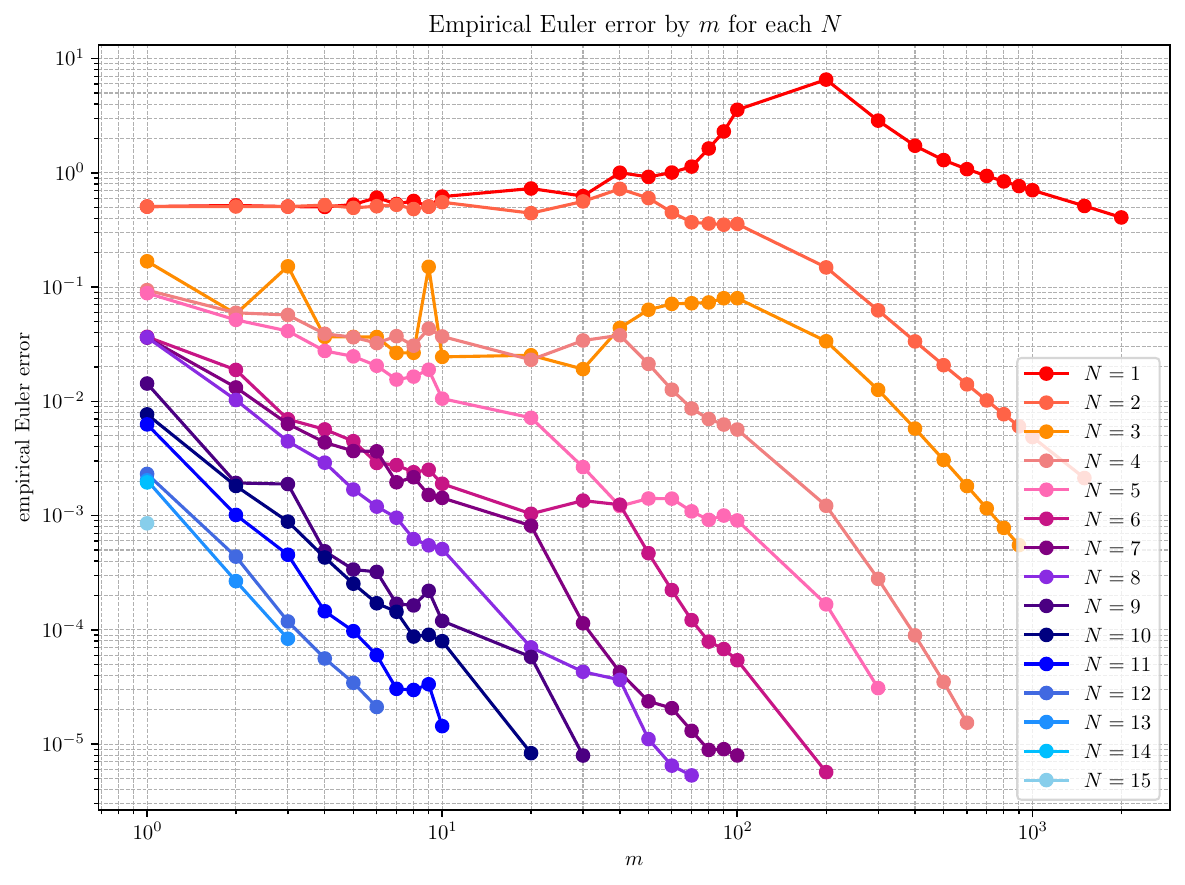}
	\end{minipage}\hspace{0.006\textwidth}%
	\begin{minipage}[t]{0.378\textwidth}
		\centering
		\includegraphics[width=\linewidth]{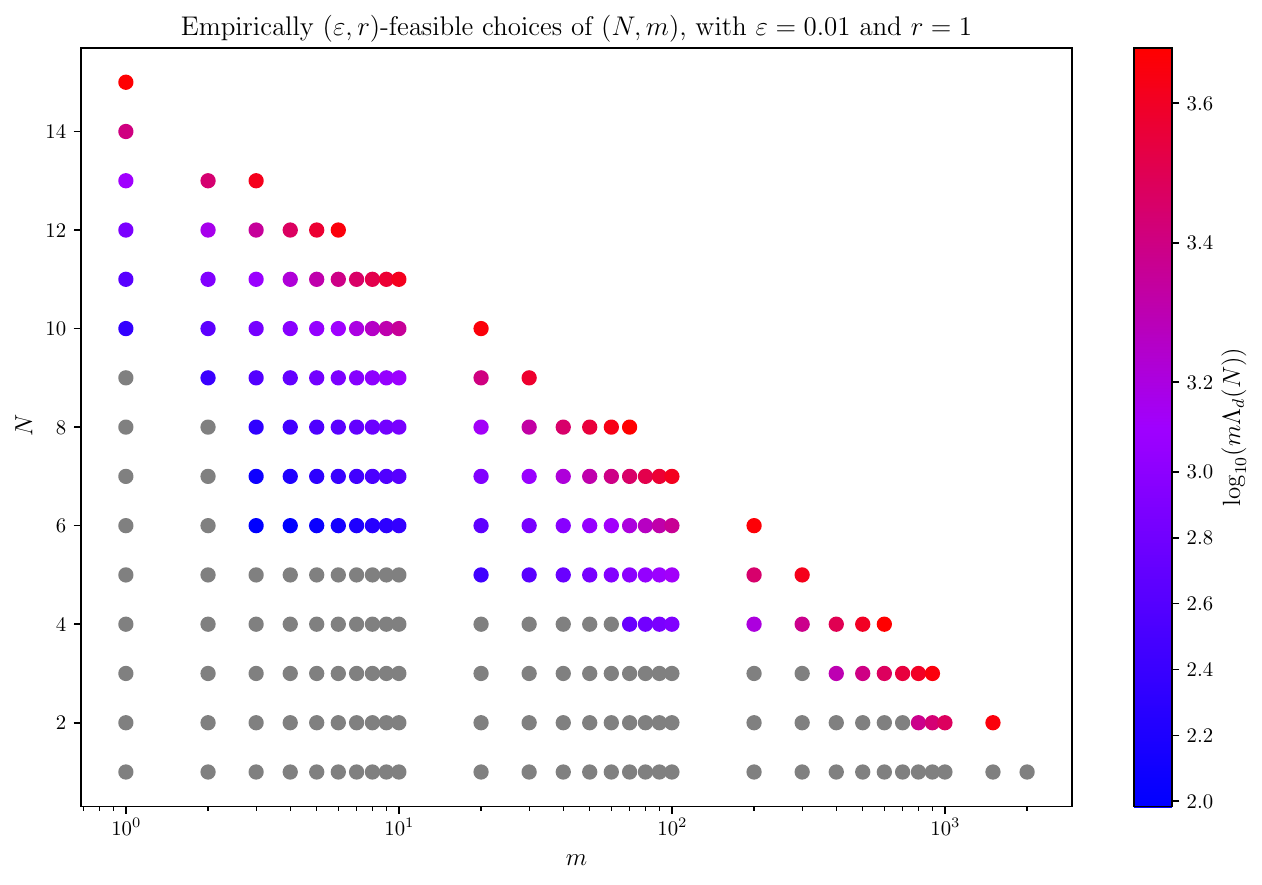}
	\end{minipage}
	\caption{We solve linear CDEs driven by the ``approximate pure figure $8$ path'' (with $\varrho = 3$), plotted on the left with the time parametrisation on the vertical axis. Averaging over $10$ matrices with norm $1$, we compute empirical Euler errors for choices of $(N,m)$ such that $m\Lambda_d(N) \leq 5000$. We see that low choices of $N$ struggle to produce convergence in the range of $m$ considered. We then plot all choices of $(N, m)$ and consider which are $(\eps, r)$-feasible, visualising for each the storage cost on a colour gradient. For this particular example, the top three choices of $(N,m)$, for $r = 1$ and $\eps \in \{10^{-2}, 10^{-3}, 10^{-4}\}$ all have $N \geq 6$. We note that, for this and later figures, taking the maximum error over equations $A$ would be more natural than averaging; this however produces plots that are harder to interpret. We speculate this is because, for small batches of equations, the worst case depends heavily on $N$ and $m$, and in the stochastic case on the chosen sample path.}
	\label{fig:your-label}
\end{figure}

\section{Paths sampled from a stochastic process}\label{sec:random}

Let $\Omega$ be a probability space and $X \colon \Omega \times [0,T] \to \bbR^d$ be a stochastic process with sample paths that are regular enough so that the signature $\Sig(X)$ can be made sense of (e.g.\ as limit in $L^2$ of the Stieltjes signatures of the piecewise linear approximations along a sequence of partitions with vanishing mesh size). This is the case classicaly for Brownian motion and also for fractional Brownian motion with Hurst parameter $H > \frac 14$. A natural reformulation of the problem \eqref{eq:problemBasic} in this context is
\begin{equation}\label{eq:problemStoch}
	\begin{split}
		\text{minimise}\quad & m\Lambda_d(N) \quad \text{over } N, m \in \mathbb N^+ \\
		\text{subject to}\quad & \|\Phi_{0,T}^A(X) - \mathcal E_N^A(X)_\pi \|_{L^2} \leq \varepsilon \quad \text{for all } A \text{ with } |A| \leq r.
	\end{split}
\end{equation}
If we also consider a drift component, noise and time should be treated inhomogeneously, and consequently the meaning $|A|$ should be adjusted, see \autoref{thm:L2error} below. The differences with the deterministic setting are that the error is computed in $L^2$, that the partition must be the same across all sample paths of $X$, and that the constraint may not take into account properties of sample paths of $X$ such as its length, since for many explicit examples of processes we expect a good $L^2$ error estimate to only involve $T$ and $d$.

Motivated by this, we will derive an $L^2$-Euler error estimate similar to the ones in previous sections for linear SDEs with drift. The convergence of these schemes is already known with sharp rate $\smash{m^{-\frac N2}}$ \cite{KP10} (which we observe is strictly faster than the $2$-rough path rate $\smash{m^{1 - \frac{N+1}{p}}}$ for $p > 2$, or even for the forbidden $p = 2$), but as in the deterministic cases focusing on the linear case enables us to additionally show factorial decay with $m$ fixed. Let $X$ be Brownian motion with augmented with drift $X^0_t = t$. Consider the linear CDE \eqref{eq:linear} and write $A_0$ for the drift term and $A_1$ for the diffusion:
\[
A = [A_0 | A_1], \qquad A_0 \in \mathcal L(\bbR^e, \bbR^e),\quad A_1 \in \mathcal L(\bbR^d, \mathcal L(\mathbb R^e, \bbR^e)),\qquad |A_0| \leq r_0, \quad |A_1| \leq r_1.
\]
Define, cf.\ \cite[(6.2) p.360]{KP10}
\begin{equation}\label{eq:AN}
\mathcal A_N \coloneqq \{(i,k) : 2i + k \leq N \text{ or } k = 0, 2i \leq N + 1 \};
\end{equation}
and note that the second condition only comes into play when $N$ is odd. We will formulate the following result in terms of linear It\^o SDEs with drift, which are equivalent to linear Stratonovich ones up to modifying the drift. Write $\Sig^{(i,k)}$ for the sum of It\^o signature components comprising $i$ drift terms and $k$ noise terms and use a similar notation for $A^{\otimes (i,k)}$.
\begin{equation}
\mathcal E_{N;{s,t}}^A \coloneqq \sum_{(i,k) \in \mathcal A_N} A^{\otimes (i,k)} \Sig^{(i,k)}_{s,t}.
\end{equation}

\begin{theorem}[$L^2$ Euler error for linear SDEs]\label{thm:L2error}
Defining 
\[
r \coloneqq \sqrt 2 r_0 \vee 2r_1^2
\]
we have
\[
\|\Phi_{0,T} - \mathcal E_{N;\pi[t_0,t_m]}\|_{L^2} \leq \exp(8Tr(d+2)+3/2)\frac{(8Tr(d+2))^{\frac{N+1}{2}}}{m^{\frac{N}{2}}\sqrt{(N+1)!}} 
\]
\end{theorem}
\begin{proof}

By orthogonality of distinct Wiener chaoses $|\mathbb E\Sig^{(i,k)}_{s,t} \otimes \Sig^{(j,l)}_{s,t}| = 0$ if $k \neq l$ and \cite[4.14]{orth} (and using the notation therein for the inner product $(\,\cdot\,,\,\cdot\,)_{\widehat \shuffle}$ on the quasi-shuffle algebra)
\begin{align*}
&|\mathbb E\Sig^{(i,k)}_{s,t}\otimes\Sig^{(j,k)}_{s,t}| \\
\leq{} &\sum_{\substack{i_0 + \ldots + i_k = i \\ j_0 + \ldots + j_k = j}} \sum_{\alpha_r,\beta_r = 1}^d (\texttt{0}^{i_0}\alpha_1 \texttt{0}^{i_1} \ldots \texttt{0}^{i_{k-1}}\alpha_k \texttt{0}^{i_k}, \texttt{0}^{j_0}\beta_1 \texttt{0}^{j_1} \ldots \texttt{0}^{j_{k-1}}\beta_k \texttt{0}^{j_k})_{\widehat \shuffle} \\
={} &\sum_{\substack{i_0 + \ldots + i_k = i \\ j_0 + \ldots + j_k = j}} \sum_{\alpha_r = 1}^d (\texttt{0}^{i_0}\alpha_1 \texttt{0}^{i_1} \ldots \texttt{0}^{i_{k-1}}\alpha_k \texttt{0}^{i_k}, \texttt{0}^{j_0}\alpha_1 \texttt{0}^{j_1} \ldots \texttt{0}^{j_{k-1}}\alpha_k \texttt{0}^{j_k})_{\widehat \shuffle} \\
={} & \frac{(t-s)^{i+j+k}}{(i+j+k)!} d^k \sum_{\substack{i_0 + \ldots + i_k = i \\ j_0 + \ldots + j_k = j}} \prod_{r = 0}^k \binom{i_r + j_r}{i_r} \\
\leq{} & \frac{(t-s)^{i+j+k}}{(i+j+k)!} d^k \sum_{\substack{i_0 + \ldots + i_k = i \\ j_0 + \ldots + j_k = j}} \prod_{r = 0}^k 2^{i_r + j_r} \\
\leq{} & \frac{(t-s)^{i+j+k}}{(i+j+k)!} d^k 2^{i + j} \big|\{(i_0, \ldots, i_k) \in \mathbb N^{k+1} : i_0 + \ldots +i_k = i\}\big|  \big|\{(j_0, \ldots, j_k) \in \mathbb N^{k+1} : j_0 + \ldots +j_k = j\}\big| \\
={} & \frac{(t-s)^{i+j+k}}{(i+j+k)!} d^k 2^{i + j} \binom{i+k}{i}\binom{j+k}{j} \\
\leq{} & \frac{(4(t-s))^{i+j+k}}{(i+j+k)!} d^k
\end{align*}
where we have used stars and bars. We have
\begin{align*}
|A^{\otimes(i,k)}|^2 = \sum_{\text{$i$ \ \texttt{0}s}} |A_{\epsilon_1}|^2 \cdots |A_{\epsilon_{i+k}}|^2 \leq \binom{i+k}{i} r_0^{2i} r_1^{2k} \leq 2^{i+k}r_0^{2i} r_1^{2k}
\end{align*}	
We now estimate various terms that will appear in the recursion. For the following two estimates the extra condition for the odd case in \eqref{eq:AN} is not needed.
\begin{align*}
	\mathbb E |\Phi_{s,t}|^2 &= \mathbb E \Big| \sum_{i,k \geq 0} A^{\otimes (i,k)}     \Sig^{(i,k)}_{s,t} \Big|^2 \\
	&= \mathbb E \sum_{i,j,k,l \geq 0}  (A^{\otimes(i,k)}    \Sig^{(i,k)}_{s,t}) \cdot (A^{\otimes(j,l)}    \Sig^{(j,l)}_{s,t}) \\
	&\leq  \sum_{i,j,k \geq 0} |A^{\otimes (i,k)}| |A^{\otimes (j,k)}| |\mathbb E \Sig^{(i,k)}_{s,t} \otimes \Sig^{(j,k)}_{s,t}| \\
	&\leq  \sum_{i,j,k \geq 0} 2^{(i+j)/2 + k} r_0^{i+j} r_1^{2k}\frac{(4(t-s))^{i+j+k}}{(i+j+k)!} d^k \\
	&\leq  \sum_{i,j,k \geq 0} \frac{(4(t-s)r)^{i+j+k}}{(i+j+k)!} d^k \\
	&\leq  \sum_{i,j,k \geq 0} \frac{(4(t-s)r)^{i+j+k}}{i!j!k!} d^k \\
	&=   e^{4(t-s)r}e^{4(t-s)r}e^{4(t-s)dr} \\
	&=   \exp(4(t-s)r(d+2))
\end{align*}
The one step Euler error estimate is argued similarly
\begin{align*}
	\mathbb E|\Phi_{s,t} -\cE_{s,t}|^2 &\leq  \sum_{\substack{2i + k \geq N + 1 \\ 2j + k \geq N + 1}} \frac{(4(t-s)r)^{i+j+k}}{(i+j+k)!} d^k \\
	&\leq  \sum_{i+j+k \geq N+1} \frac{(4(t-s)r)^{i+j+k}}{i!j!k!} d^k \\
	&=  \sum_{n = N+1}^\infty \frac{(4(t-s)r(d+2))^n}{n!} \\
	&\leq  \frac{(4(t-s)r(d+2))^{N+1}}{(N+1)!}\exp(4(t-s)r(d+2)).
\end{align*}
We also need to consider the cross term
\begin{align*}
	|\mathbb E[(\Phi_{s, t} - \mathcal E_{s,t}) \cdot \mathcal E_{s,t}]|
	&\leq  \sum_{\substack{(i,k) \in \mathcal A_N^\complement \\ (j,k) \in \mathcal A_N}} |A^{\otimes (i,k)}||A^{\otimes(j,k)}||\mathbb E \Sig^{(i,k)}_{s,t} \otimes \Sig^{(j,k)}_{s,t}| \\
	&\leq  \sum_{\substack{(i,k) \in \mathcal A_N^\complement \\ j \geq 0}} \frac{(4(t-s)r)^{i+j+k}}{i!j!k!} d^k \\
	&\leq \begin{dcases}
		\frac{(4(t-s)r(d+2))^{\frac N2+1}}{\big(\frac N2+1\big)!}\exp(4(t-s)r(d+2)) &N \text{ even} \\
		\frac{(4(t-s)r(d+2))^{\frac{N+1}2+1}}{\big(\frac{N+1}2+1\big)!}\exp(4(t-s)r(d+2)) &N \text{ odd.}
	\end{dcases}
\end{align*}
Here we have used that if $N$ is even, disregarding the second condition in \eqref{eq:AN}, $2(i+k) \geq 2i + k \geq N+1$ implies $i+k \geq \frac N2 + 1$. If $N$ is odd, instead, the same implication only gives $i + k \geq \frac{N+1}2$, with equality only possible if $k = 0$ since $i = (2i+k) - (i+k) \geq N+1 - \frac{N+1}{2} = \frac{N+1}{2}$, precisely the case covered by the extra condition.

In the following recursion we use independence of Brownian increments. Recall the weighted Young inequality: for $a,b \geq 0$ and $\eps > 0$ (we will use it with $\eps = m^{-1}$)
\[
2ab \leq \frac{a^2}{\eps} + \eps b^2.
\]
and the simple form of the discrete Grönwall inequality \eqref{eq:discreteGronwall1}.
\begin{align*}
	&\|\Phi_{t_0,t_k} - \mathcal E_{\pi[t_0,t_k]}\|_{L^2}^2 \\
	={}& \|(\Phi_{t_{k-1}, t_k} - \cE_{t_{k-1}, t_k})\Phi_{t_0,t_{k-1}}\|_{L^2}^2 + \|\mathcal E_{t_{k-1},t_k}(\Phi_{t_0,t_{k-1}} - \mathcal E_{\pi[t_0,t_{k-1}]})\|_{L^2}^2 \\
	&+ 2 \mathbb E\big[\Phi^\intercal_{t_0,t_{k-1}} \mathbb E[(\Phi_{t_{k-1}, t_k} - \mathcal E_{t_{k-1},t_k})^\intercal \mathcal E_{t_{k-1},t_k}] (\Phi_{t_0,t_{k-1}} - \mathcal E_{\pi[t_0,t_{k-1}]}) \big] \\
	\leq{} &\mathbb E |\Phi_{t_{k-1}, t_k} - \cE_{t_{k-1}, t_k}|^2 \mathbb E |\Phi_{t_0,t_{k-1}}|^2 + \mathbb E|\mathcal E_{t_{k-1},t_k}|^2 \|\Phi_{t_0,t_{k-1}} - \mathcal E_{\pi[t_0,t_{k-1}]}\|_{L^2}^2 \\
	&+ 2 \|\Phi_{t_0,t_{k-1}} \|_{L^2} \|\Phi_{t_0,t_{k-1}} - \mathcal E_{\pi[t_0,t_{k-1}]}\|_{L^2}  | \mathbb E [(\Phi_{t_{k-1}, t_k} - \mathcal E_{t_{k-1},t_k}) \cdot \mathcal E_{t_{k-1},t_k}]| \\
	\leq{} &\exp(4t_k r(d+2))\frac{(4 T r(d+2))^{N+1}}{m^{N+1}(N+1)!} + \exp(4T r(d+2)/m) \|\Phi_{t_0,t_{k-1}} - \mathcal E_{\pi[t_0,t_{k-1}]}\|_{L^2}^2 \\
	&+2\exp(4t_k r(d+2))
	\left\{
	\begin{array}{ll}
		\displaystyle
		\frac{(4Tr(d+2))^{\frac N2+1}}
		{m^{\frac N2+1}\big(\frac N2+1\big)!},
		& N \text{ even}\\[1.2em]
		\displaystyle
		\frac{(4Tr(d+2))^{\frac{N+3}{2}}}
		{m^{\frac{N+3}{2}}\big(\frac{N+3}{2}\big)!},
		& N \text{ odd}
	\end{array}
	\right\}
	\|\Phi_{t_0,t_{k-1}} - \mathcal E_{\pi[t_0,t_{k-1}]}\|_{L^2} \\
	\leq{}&
	\exp(4t_k r(d+2))
	\frac{(4 T r(d+2))^{N+1}}{m^{N+1}(N+1)!} + \left(\exp(4T r(d+2)/m)+\frac1m\right)
	\|\Phi_{t_0,t_{k-1}}-\mathcal E_{\pi[t_0,t_{k-1}]}\|_{L^2}^2
	\\
	&+\exp(8t_k r(d+2))
	\left\{
	\begin{array}{ll}
		\displaystyle
		\frac{(4Tr(d+2))^{N+2}}
		{m^{N+1}\big(\frac N2+1\big)!^2},
		& N \text{ even}\\[1.2em]
		\displaystyle
		\frac{(4Tr(d+2))^{N+3}}
		{m^{N+2}\big(\frac{N+3}{2}\big)!^2},
		& N \text{ odd}
	\end{array}
	\right\}\\
	\leq{}&
	\left(\exp(4T r(d+2)/m)+\frac1m\right)
	\|\Phi_{t_0,t_{k-1}}-\mathcal E_{\pi[t_0,t_{k-1}]}\|_{L^2}^2\\
	&+
	\left[
	\exp(4t_k r(d+2))
	+
	\exp(8t_k r(d+2))
	\left\{
	\begin{array}{ll}
		\displaystyle
		4Tr(d+2), & N \text{ even}\\[0.8em]
		\displaystyle
		(4Tr(d+2))^2, & N \text{ odd}
	\end{array}
	\right\}
	\right]
	\frac{(8Tr(d+2))^{N+1}}{m^{N+1}(N+1)!}\\
	\leq{}&
	\exp\!\left(\frac{4Tr(d+2)+1}{m}\right)
	\|\Phi_{t_0,t_{k-1}}-\mathcal E_{\pi[t_0,t_{k-1}]}\|_{L^2}^2\\
	&+
	\exp(8t_k r(d+2))
	\bigl(1+4Tr(d+2)+(4Tr(d+2))^2\bigr)
	\frac{(8Tr(d+2))^{N+1}}{m^{N+1}(N+1)!}\\
	\leq{}&
	\sum_{i=1}^{k}
	\exp(8t_i r(d+2))
	\bigl(1+4Tr(d+2)+(4Tr(d+2))^2\bigr)
	\frac{(8Tr(d+2))^{N+1}}{m^{N+1}(N+1)!}
	\exp\!\left(\frac{(k-i)(4Tr(d+2)+1)}{m}\right)\\
	\leq{}&
	k\exp(16Tr(d+2)+3)
	\frac{(8Tr(d+2))^{N+1}}{m^{N+1}(N+1)!}.
\end{align*}
Note that if we had not required the second condition in \eqref{eq:AN}, the odd case would only have yielded a mean square rate of $m^{-(\frac{N+1}{2} \cdot 2 - 1 - 1)} = m^{-(N-1)}$ (which in particular fails to converge in the Euler-Maruyama case $N = 1$). We have used the bounds $\frac{(N+1)!}{(\frac N2 + 1)!^2} \leq 2^{N+1}$ for $N$ even and $\frac{(N+1)!}{(\frac{N+3}{2})!^2} \leq 2^{N+1}$ for $N$ odd, as well as the elementary bounds $e^{a/m} + 1/m \leq e^{a/m}(1+1/m) \leq e^{(a+1)/m}$ and $1+x+x^2 \leq e^{x+2}$. Taking $k = m$ finishes the proof.
\end{proof}

We briefly study the analogue of the relaxed constrained optimisation problem \eqref{eq:relaxed} in the Brownian/$L^2$ setting. We can replace the error estimate in the constraint with the bound from \autoref{thm:L2error}
\[
\| \Phi_{0,T} - \cE_{N;\pi} \|_{L^2} \leq B e^R \frac{R^{\frac{N+1}{2}}}{m^{N/2} \sqrt{(N+1)!}},\qquad R = R(r,d,T) \coloneqq CTr(d+2),
\]
with universal constants $B, C$. For the purposes of the size of the log-signature, we can consider $d+1$ as the dimension of the path, even though this slightly overcounts the dimension due to the parabolic scaling (also note that the second condition in \eqref{eq:AN} does not entail the saving of extra signature terms, since it only affects the universal terms $\int_{\Delta^k[s,t]} \d u_1 \cdots \d u_k = \frac{(t-s)^k}{k!}$). Reasoning analogously to the $1$-variation case, we obtain
\begin{proposition}
\begin{align*}
	N^* &\sim \sqrt{\frac{\lambda(\eps,R)}{\log(d+1)}}, \qquad m^* \sim R \sqrt{\frac{\log (d+1)}{\lambda(\eps,R)}} \exp\big( \sqrt{\lambda(\eps,R) \log(d+1)} \big) \\
	\text{with } \lambda(\eps, R) &= \log \frac{B^2Re^{2R}}{\eps^2}
\end{align*}
as $r \to \infty$ or $\eps \to 0^+$ with the other parameters fixed.
\end{proposition}
These asymptotics are structurally similar to those derived for the bounded variation case, the main differences lying the squared exponential and $\eps$ in the definition of $\lambda$, as well as the dependence of $R$ on $d$, a consequence of the fact that Brownian motion depends on the ambient dimension.

\begin{remark}
It should be pointed out that computing the signature of Brownian motion from the time series at a reasonable accuracy is challenging, since the $m^{-\frac 12}$ rate predicted by the Euler scheme is the best possible \cite{CC80}. However, since \autoref{thm:L2error} largely relies on martingale techniques, we believe it should be possible to prove a similar bound for the piecewise linear approximations of Brownian motion, which are always what is used in applications (e.g.\ below).
\end{remark}

We conclude with an empirical study of $d$-dimensional fractional brownian motion with Hurst parameter $H$ ($H$-fBm), a self-similar Gaussian process with stationary increments that has received a lot of attention in the literature, see \cite{CQ02} for the rough path lift of this process with $d = 4$ and $H > \frac 14$. We sample $100$ paths i.i.d.\ from this measure with $H = 0.3, 0.5, 0.7, 0.9$, and $10$ random matrices of norm $r = 1$. We use $10^4$ grid points on the interval $[0,1]$ for sampling the noise, and we further discretise $10$-fold for computing the signatures, which guarantees we are computing the Stratonovich signatures. When applying the Euler scheme for $N = 1,2$, we additionally add the remaining terms in the level-$3$ Taylor expansion
\[
\sum_{j = N+1}^3 \frac{1}{j!}A^{\otimes j}Y_{t_{k-1}}X_{t_{k-1},t_k}^{\otimes j}:
\]
for Brownian motion adding the term for $j = 2$ is equivalent to converting to It\^o form and applying the ordinary Euler scheme, and for $\frac 14 < H < \frac 12$ is necessary for convergence (while for $H > \frac 12$ is still needed for the optimal rate) \cite{BFRS16, HLN16}; note that this does not require any extra signature terms and therefore does not impact storage considerations. Our numerical results show that, for several \say{typical} choices of the parameters, taking high values of $N$, often higher than the intrinsic roughness $\lfloor H^{-1} \rfloor$ which is qualitatively necessary for rough path analysis, dramatically improves storage cost.

\begin{figure}[h!]
	\centering
	\setlength{\tabcolsep}{1pt}
	\renewcommand{\arraystretch}{0}
	
	\begin{tabular}{@{}ccc@{}}
		\begin{minipage}{0.325\textwidth}
			\centering
			\includegraphics[width=\linewidth]{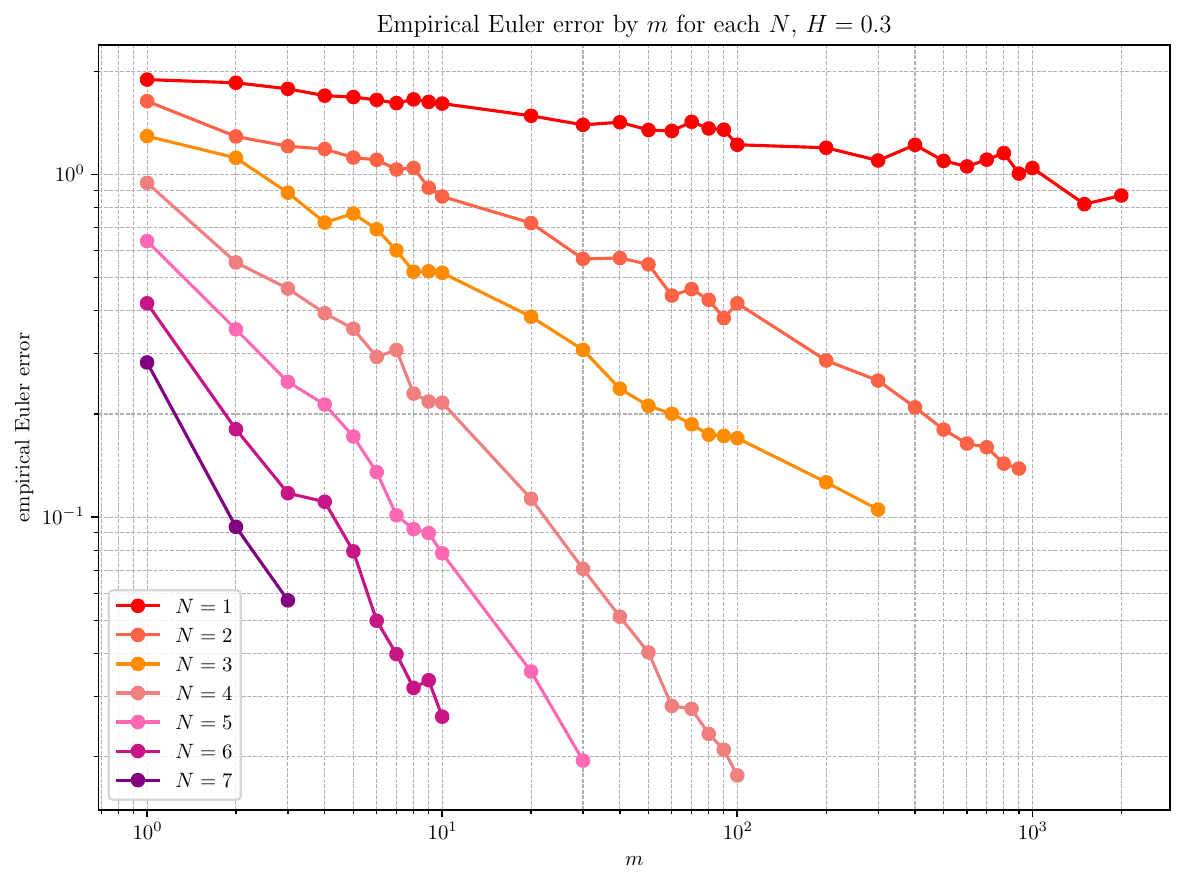}
		\end{minipage} &
		\begin{minipage}{0.325\textwidth}
			\centering
			\includegraphics[width=\linewidth]{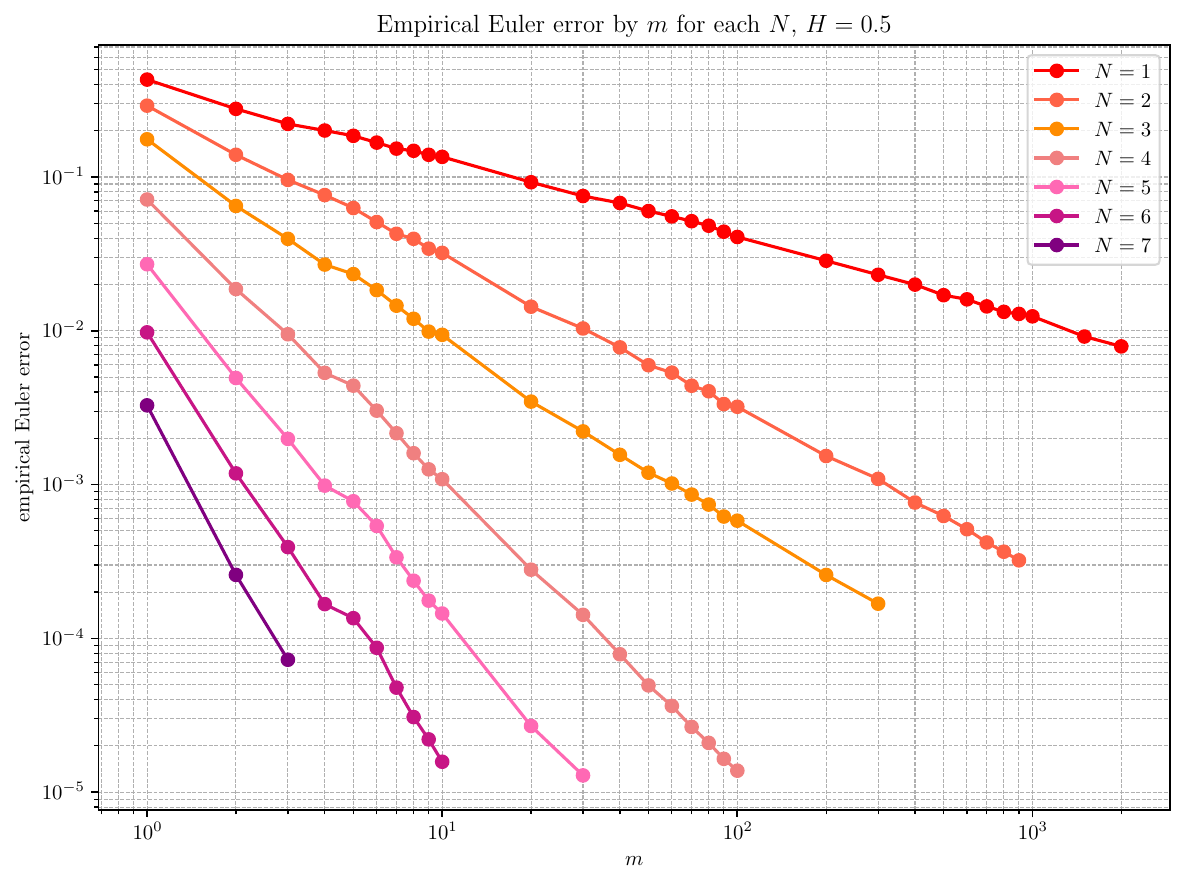}
		\end{minipage} &
		\begin{minipage}{0.325\textwidth}
			\centering
			\includegraphics[width=\linewidth]{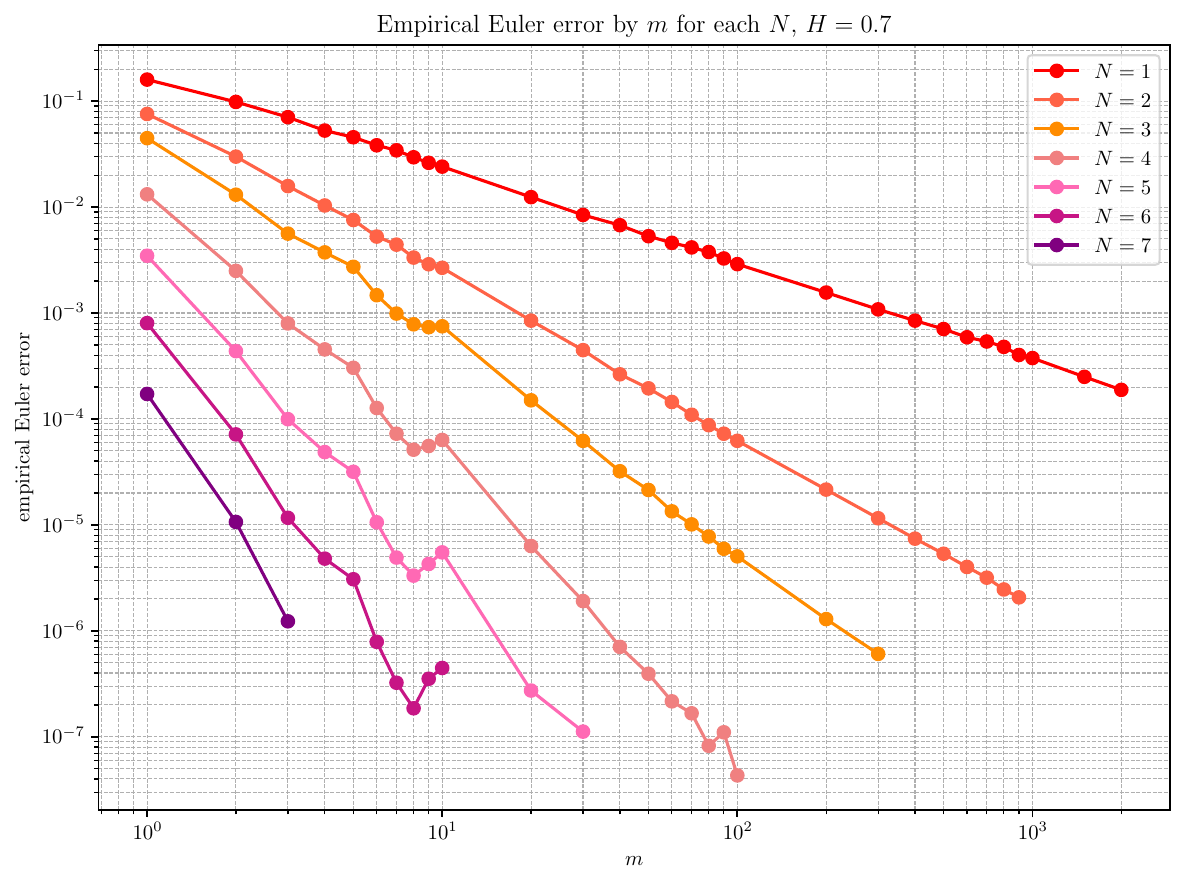}
		\end{minipage}
		\\[-0.1em]
		\begin{minipage}{0.325\textwidth}
			\centering
			\includegraphics[width=\linewidth]{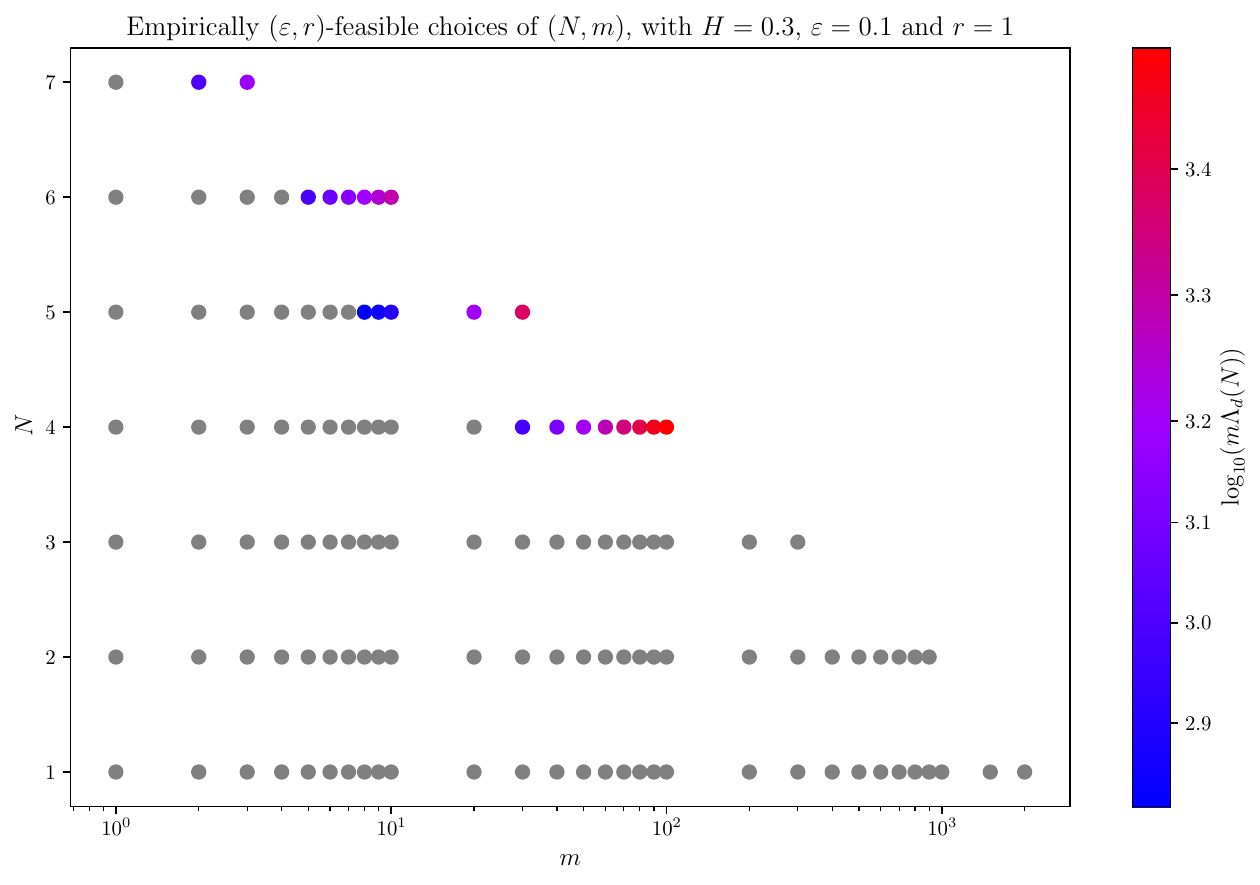}
		\end{minipage} &
		\begin{minipage}{0.325\textwidth}
			\centering
			\includegraphics[width=\linewidth]{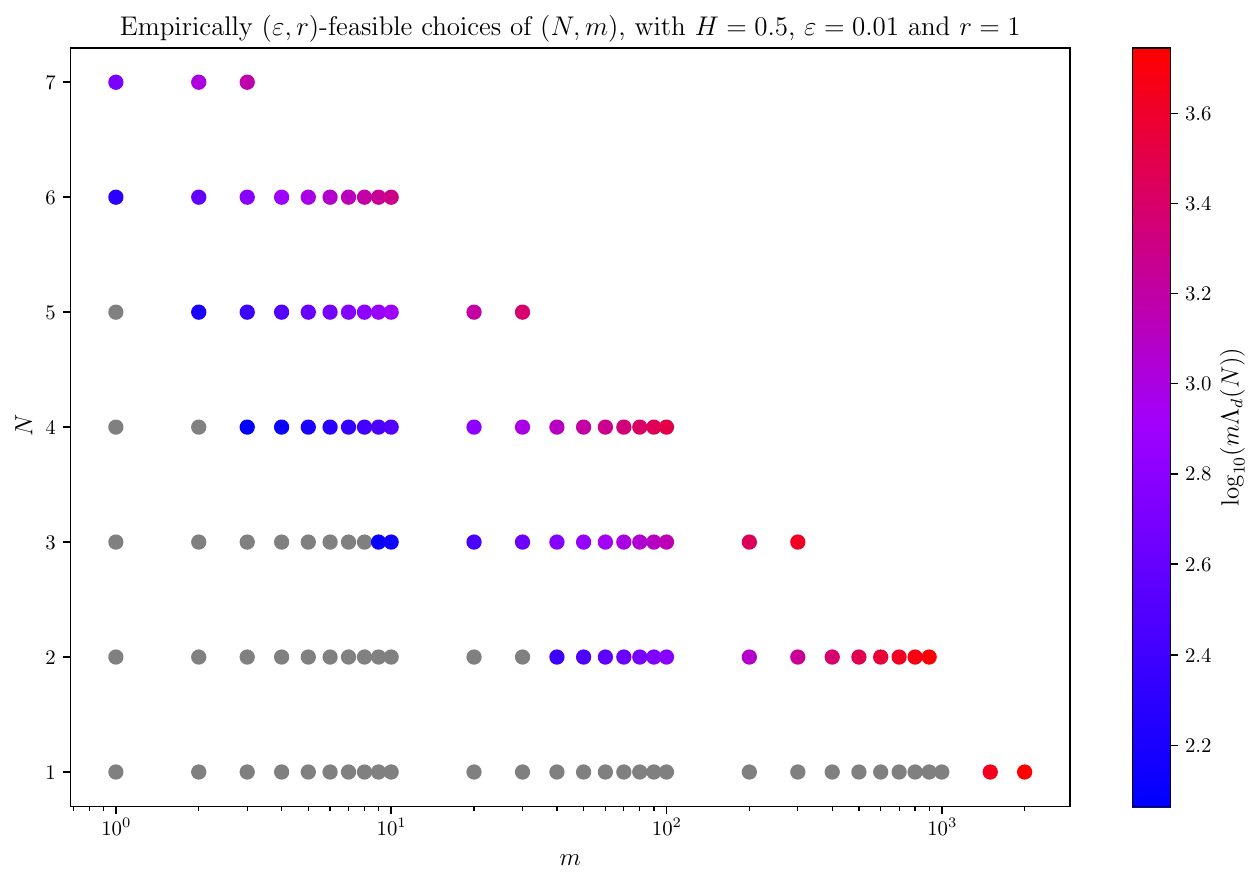}
		\end{minipage} &
		\begin{minipage}{0.325\textwidth}
			\centering
			\includegraphics[width=\linewidth]{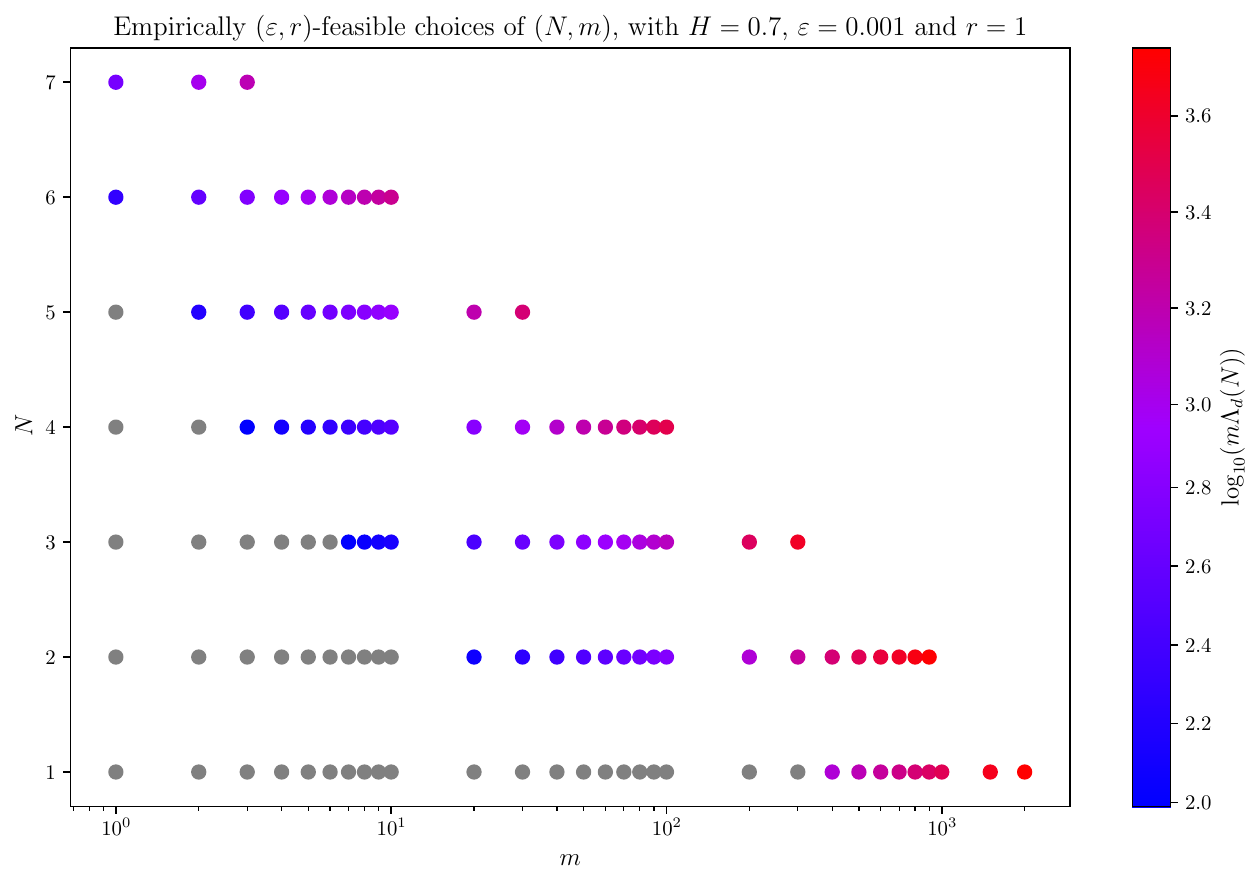}
		\end{minipage}
	\end{tabular}
	
	\vspace{-0.5em}
	\caption{Error plots and storage cost for $(\eps,r)$-admissible choices of $(N,m)$ with $m \Lambda_d(N) \leq 10000$, for $r = 1$ and $(H, \eps) \in \{(0.3,0.1), (0.5, 0.01), (0.7,0.001)\}$.}
	\label{fig:six_figures}
\end{figure}

\begin{figure}[h!]
	\centering
	\setlength{\tabcolsep}{1pt}
	\renewcommand{\arraystretch}{0}
	
	```
	\begin{tabular}{@{}ccc@{}}
		\begin{minipage}{0.32\textwidth}
			\centering
			\includegraphics[width=\linewidth]{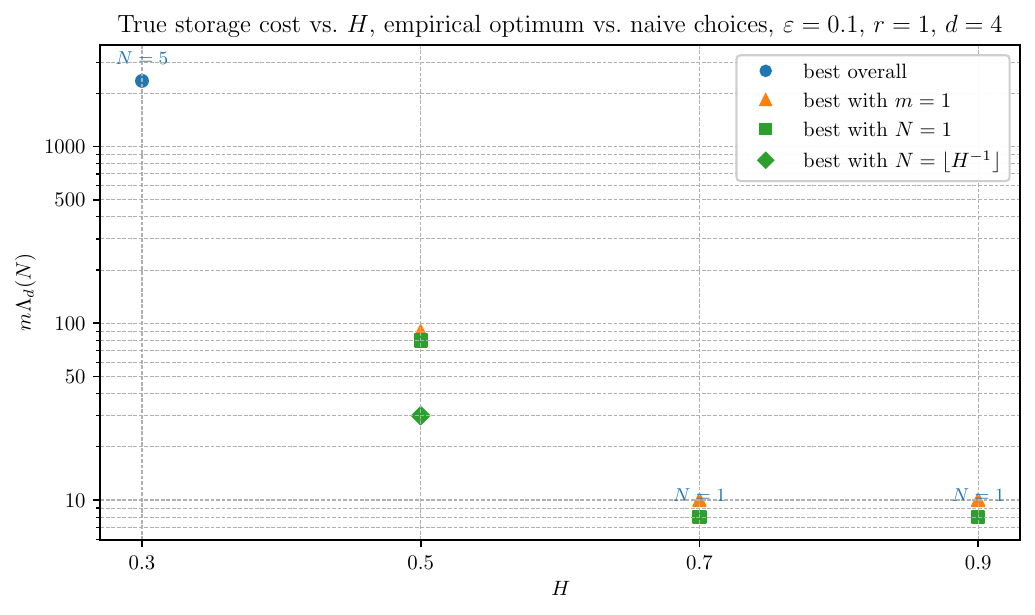}
		\end{minipage} &
		\begin{minipage}{0.32\textwidth}
			\centering
			\includegraphics[width=\linewidth]{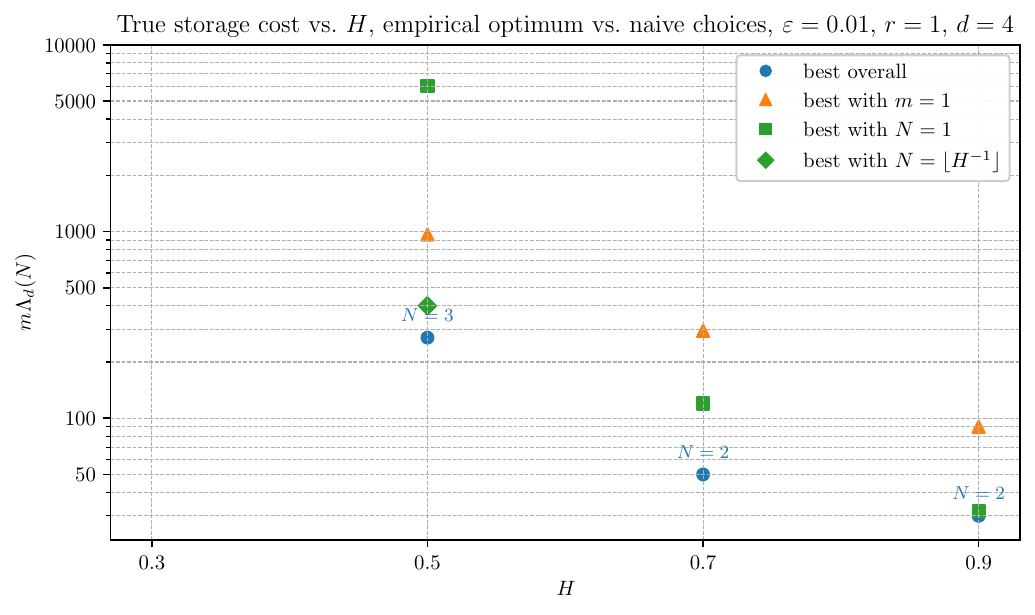}
		\end{minipage} &
		\begin{minipage}{0.32\textwidth}
			\centering
			\includegraphics[width=\linewidth]{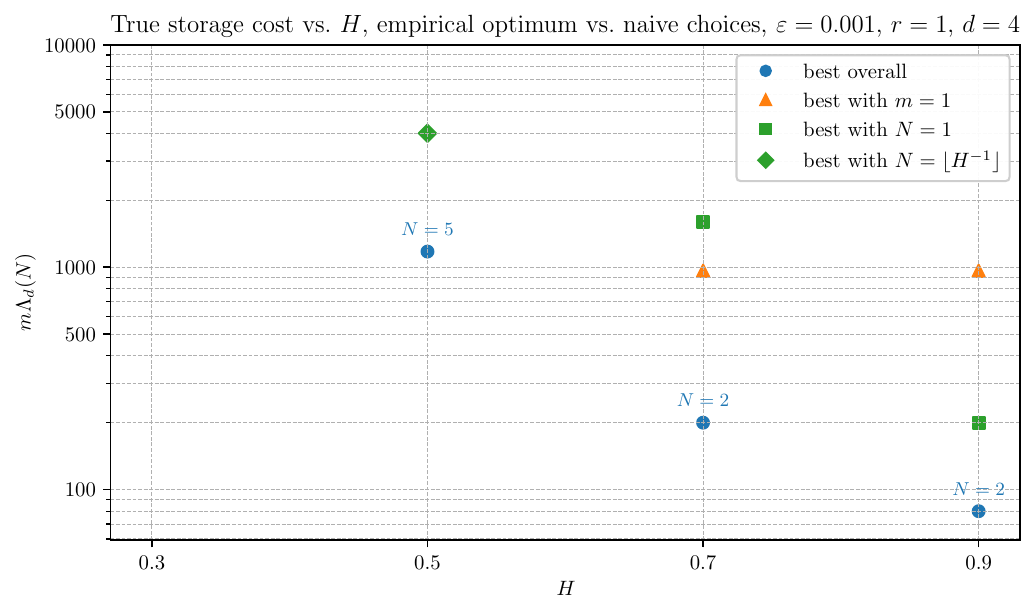}
		\end{minipage}
		\\[-0.1em]
		\begin{minipage}{0.32\textwidth}
			\centering
			\includegraphics[width=\linewidth]{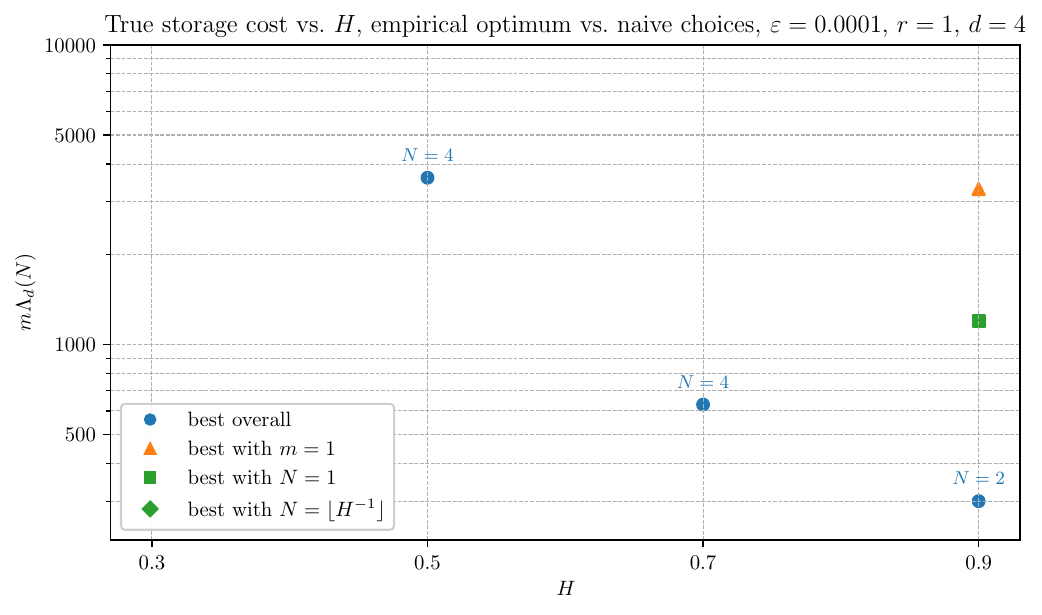}
		\end{minipage} &
		\begin{minipage}{0.32\textwidth}
			\centering
			\includegraphics[width=\linewidth]{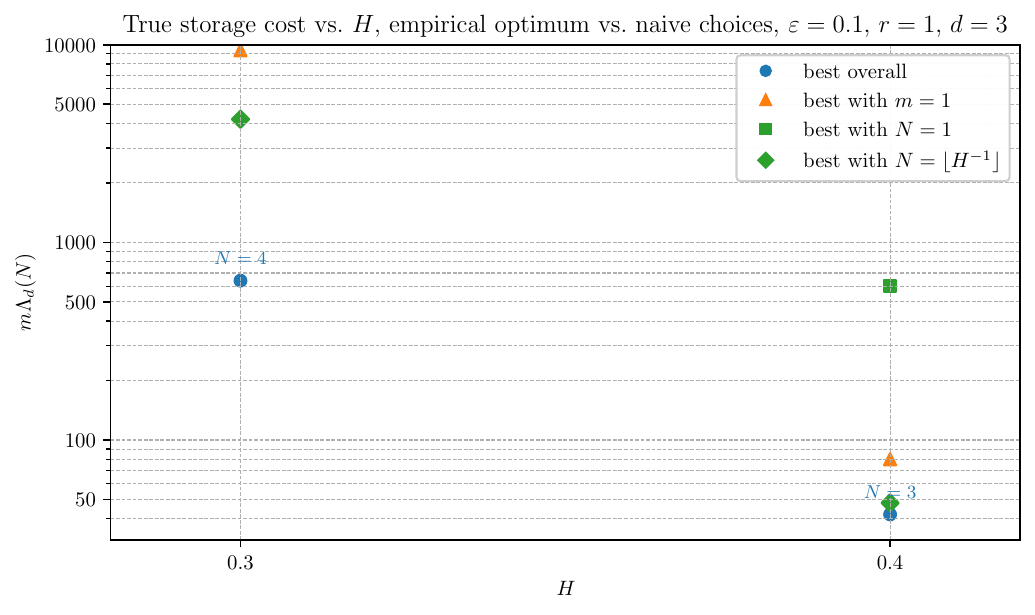}
		\end{minipage} &
		\begin{minipage}{0.32\textwidth}
			\centering
			\includegraphics[width=\linewidth]{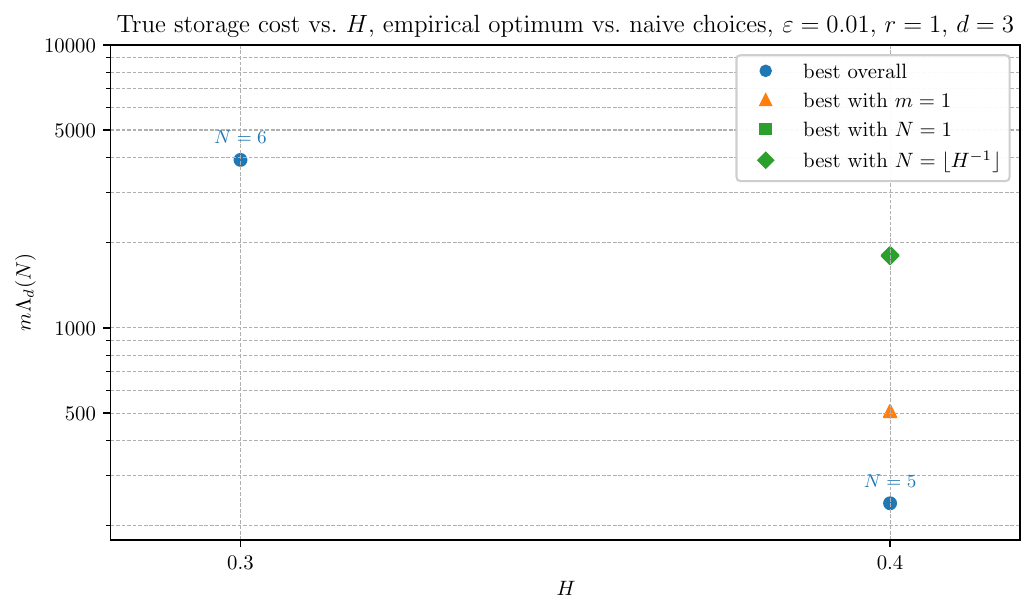}
		\end{minipage}
	\end{tabular}
	
	\vspace{-0.5em}
	\caption{Storage cost plotted vs.\ $H$ for $(\eps,r)$-admissible choices of $(N,m)$ with $m \Lambda_d(N) \leq 10000$. We consider both $d = 4$ for a broader range of $H$ and in the last two plots focus on $d = 3$ for $H = 0.3,0.4$ (lowering the dimension with fixed storage budget allows for higher $N, m$). We plot both the empirical optimum vs.\ that of the naive choices with $N = 1$ or $m = 1$, as well as the choice $N = \lfloor H^{-1} \rfloor$ which is important from the point of view of rough path theory. If a label is not present it means no $(\eps,r)$-admissible choice of $(N,m)$ is present.}
	\label{fig:four_figures}
\end{figure}

\bibliographystyle{alpha} 
\bibliography{refs}\addcontentsline{toc}{chapter}{References}

\end{document}